\documentclass[twoside, 12pt]{amsart}
\usepackage{amsmath}
\usepackage{amsfonts}
\usepackage{calrsfs} 
\usepackage{amssymb}
\usepackage{graphicx}
\usepackage[all,cmtip]{xy} 
\usepackage{fancyhdr}
\usepackage[margin=1.1in]{geometry} 

\usepackage{enumitem}

\usepackage{aliascnt} 

\theoremstyle{plain}
\newtheorem{theorem}{Theorem}[section]
\newaliascnt{corollary}{theorem}

\aliascntresetthe{corollary}

\newaliascnt{lemma}{theorem}
\newtheorem{lemma}[lemma]{Lemma}
\aliascntresetthe{lemma}
\newaliascnt{proposition}{theorem}
\newtheorem{proposition}[proposition]{Proposition}
\aliascntresetthe{proposition}

\newaliascnt{hypotheses}{theorem}

\aliascntresetthe{hypotheses}

\theoremstyle{definition}
\newaliascnt{definition}{theorem}
\newtheorem{definition}[definition]{Definition}
\aliascntresetthe{definition}

\newaliascnt{example}{theorem}

\aliascntresetthe{example}

\newaliascnt{remark}{theorem}
\newtheorem{remark}[remark]{Remark}
\aliascntresetthe{remark}

\newaliascnt{remarks}{theorem}
\newtheorem{remarks}[remarks]{Remarks}
\aliascntresetthe{remarks}

\newcommand\shtitle{Superconnections and period domains}
\newcommand\shortauthor{Luis E. Garcia}

\numberwithin{equation}{section}

\usepackage[pdftex,breaklinks,colorlinks,
citecolor=blue,
linkcolor=blue,
urlcolor=blue]{hyperref}

\everymath{\displaystyle}

\begin{document}
\title[Title]{Superconnections, theta series, and period domains}
\author{Luis E. Garcia}
\maketitle
\begin{abstract}
We use superconnections to define and study some natural differential forms on period domains $\mathbb{D}$ that parametrize polarized Hodge structures of given type on a rational quadratic vector space $V$. These forms depend on a choice of vectors $v_1,\ldots,v_r \in V$ and have a Gaussian shape that peaks on the locus where $v_1,\ldots,v_r$ become Hodge classes. We show that they can be rescaled so that one can form theta series by summing over a lattice $L^r \subset V^r$. These series define differential forms on arithmetic quotients $\Gamma \backslash \mathbb{D}$. We compute their cohomology class explicitly in terms of the cohomology classes of Hodge loci in $\Gamma \backslash \mathbb{D}$. When the period domain is a hermitian symmetric domain of type IV, we show that the components of our forms of appropriate degree recover the forms introduced by Kudla and Millson. In particular, our results provide another way to establish the main properties of these forms.
\end{abstract}
\setcounter{tocdepth}{1}
\tableofcontents

\section{Introduction}

The goal of this paper is to introduce certain natural theta series that define closed differential forms on arithmetic quotients of period domains. We will show that the cohomology classes of these theta series are Siegel modular forms, and will determine their Fourier expansion in terms of Hodge loci. When the period domain is a Hermitian symmetric domain, we recover the main results of Kudla and Millson \cite{KudlaMillson1} concerning modularity of generating series of special cycles in the cohomology of orthogonal Shimura varieties. Our main tool will be Quillen's Chern form attached to a superconnection. It will allow us to define differential forms on period domains that generalize the forms introduced by Kudla and Millson. We now describe our results in more detail.

\subsection{Natural differential forms on period domains} Let $\mathbb{D}$ be the period domain parametrizing Hodge structures of even weight $w$ with given Hodge numbers on a fixed $\mathbb{Q}$-vector space $V_\mathbb{Q}$, polarized by a bilinear form $Q$. For any $v \in V_\mathbb{Q}$ there is a complex submanifold $\mathbb{D}_v \subset \mathbb{D}$ whose points correspond to Hodge structures where $v$ is a Hodge class. There is a Hodge bundle $\mathcal{F}=\mathcal{F}^{w/2+1}$ on $\mathbb{D}$ and a holomorphic section $s_v$ of $\mathcal{F}^\vee$ such that the Koszul complex $(\wedge \mathcal{F},s_v)$ gives a resolution of $\mathcal{O}_{\mathbb{D}_v}$ for non-zero $v$.

Using the Hodge metric on $\mathcal{F}$, we construct a superconnection $\nabla_v$ on this complex and define a closed differential form
\begin{equation}
\varphi(v) = e^{-\pi Q(v,v)} tr_s(e^{\nabla_v^2}) \in A^{2*}(\mathbb{D}).
\end{equation}
Let $G=O(V_\mathbb{R},Q)$. We show that, with respect to the natural action of $G$ on $\mathbb{D}$, the form $\varphi(\cdot)$ is $G$-invariant; that is, it satisfies 
$g^*\varphi(gv)=\varphi(v)$ for every $g \in G$. Let $\mathcal{S}(V_\mathbb{R})$ be the Schwartz space of smooth and rapidly decreasing functions on $V_\mathbb{R}$. Denote by $\omega=\omega_\psi$ the Weil representation, with respect to the additive character $\psi(x)=e^{2\pi ix}$, of the metaplectic double cover of $SL_2(\mathbb{R})$ on $\mathcal{S}(V_\mathbb{R})$.

\begin{theorem} \label{thm:thm_1}
For fixed $z \in \mathbb{D}$, the form $\varphi$ belongs to $\mathcal{S}(V_\mathbb{R}) \otimes \wedge T^*\mathbb{D}_z$. Up to explicit exact $\mathcal{S}(V_\mathbb{R})$-valued $G$-invariant forms, $\varphi$ is a lowest weight vector of weight $\tfrac{1}{2}\dim V$.
\end{theorem}
We obtain similar results for the forms $\varphi(v_1,\ldots,v_r)=\varphi(v_1) \wedge \cdots \wedge \varphi(v_r)$.

\subsection{Theta series and arithmetic quotients of period domains} Let $L \subset V_\mathbb{Q}$ be an even integral lattice with dual lattice $L^\vee \supset L$ and let
\begin{equation}
\Gamma_L=\{\gamma \in G^0 | \gamma(L)=L, \ \gamma|_{L^\vee/L}=id\}.
\end{equation}
Then $\Gamma_L$ is a discrete subgroup of $G^0$. Fix a connected component $\mathbb{D}^+$ of $\mathbb{D}$ and define
\begin{equation}
X_L= \Gamma_L \backslash \mathbb{D}^+.
\end{equation}
The complex manifolds $X_L$ were introduced by Griffiths and Schmid \cite{GriffithsSchmid} and have recently been studied by many authors; see, for example, \cite{GreenGriffithsKerrBook} for a study of their geometric and arithmetic properties. For general Hodge numbers, they are not algebraic, and in fact not even K\"ahler. There is a natural collection of analytic cycles $Hdg(n,L)$ indexed by positive integers $n$: the points in the support of $Hdg(n,L)$ correspond to $\Gamma_L$-orbits of Hodge structures on $V_\mathbb{Q}$ containing a Hodge class $v \in L$ with $Q(v,v)=2n$.

Let $\mathbb{H}$ be the complex upper half plane. Since $\varphi$ is rapidly decreasing, the theta series
\begin{equation}
\theta(\tau;L) = y^{-\dim V/4} \sum_{v \in L} \omega(g_\tau)\varphi(v) 
\end{equation}
converges for any $\tau=x+iy \in \mathbb{H}$ and defines a closed differential form on $X_L$. 

\begin{theorem} \label{thm:thm_intro_1_2}
Let $[\theta(\tau;L)]$ be the cohomology class of $\theta(\tau;L)$. Then $[\theta(\tau;L)]$ is a holomorphic modular form of weight $\tfrac{1}{2}\dim V$, valued in $H^*(X_L)$. Its Fourier expansion is given by
\begin{equation} \label{eq:thm_1.2}
\left(\frac{i}{2\pi} \right)^{rk(\mathcal{F})} \cdot [\theta(\tau;L)] \cup Td(\mathcal{F}^\vee) = c^{top}(\mathcal{F}^\vee) + \sum_{n \geq 1} Hdg(n,L) q^n, \quad q = e^{2\pi i \tau},
\end{equation}
where $Td(\cdot)$ is the Todd class and $c^{top}(\cdot)$ denotes the Chern class of top degree.
\end{theorem}
In particular, the right hand side of \eqref{eq:thm_1.2} is a holomorphic modular form of weight $\dim V/2$ valued in $H^{2 rk(\mathcal{F})}(X_L)$. (We remark that this consequence of the theorem can also be proved directly from the modularity of the generating series of special cycles in the symmetric space attached to $G$ proved in \cite{KudlaMillson3}.) 

More generally, for any positive integer $r$ and any symmetric positive semidefinite matrix $T$ with integral entries, we consider the locus $Hdg(T,L)$ in $X_L$ consisting of $\Gamma_L$-orbits of Hodge structures on $V_\mathbb{Q}$ containing Hodge classes $v_1,\ldots,v_r \in L$ such that $Q(v_i,v_j)=2T$. For $\tau$ in the Siegel half plane of genus $r$, we can define a theta series $\theta(\tau;L^r)$. We prove similarly that its cohomology class $[\theta(\tau;L^r)]$ is a $H^*(X_L)$-valued holomorphic Siegel modular form of weight $\tfrac{1}{2}\dim V$, with Fourier expansion
\begin{equation} \label{thm:thm_1_intro}
\left(\frac{i}{2\pi} \right)^{r \cdot rk(\mathcal{F})} \cdot[\theta(\tau;L^r)] \cup Td(\mathcal{F}^\vee)^r = \sum_{T \geq 0} Hdg(T,L) \cup c^{top}(\mathcal{F}^\vee)^{r-rk(T)} q^T.
\end{equation}

\subsection{Relation with Kudla-Millson forms} Suppose that the period domain $\mathbb{D}$ classifies polarized Hodge structures of weight $2$ and type $(1,n,1)$. In this case, $\mathbb{D}$ is the hermitian symmetric domain associated with the Lie group $SO(n,2)$. For a positive integer $r$, Kudla and Millson define forms
\begin{equation}
\varphi_{KM} \in [\mathcal{S}(V_\mathbb{R}^r) \otimes A^{r,r}(\mathbb{D})]^G
\end{equation}
and study their properties. These forms (defined in \cite{KudlaMillson1} for more general symmetric spaces) have proved to be of fundamental importance in the study of special cycles; see, for example, \cite{KudlaOrthogonal,KudlaMillson3} or the recent work of Bergeron, Millson and Moeglin, \cite{BergeronMillsonMoeglin}, proving the Hodge conjecture for compact arithmetic quotients of the complex $n$-ball in certain cohomological degrees. The original motivation for this paper was to understand the construction of the forms $\varphi_{KM}$. In Theorem~\ref{thm:comparison_varphi_varphi_KM} we show that
\begin{equation}
\varphi_{KM}(v_1,\ldots,v_r)= \left( \frac{i}{2\pi} \right)^r \varphi(v_1,\ldots,v_r)[2r], \quad v_1,\ldots,v_r \in V_\mathbb{R}.
\end{equation}
Thus we obtain another approach to proving (for the group $G \cong O(n,2)$) the main properties of $\varphi_{KM}$ and the associated theta series established in \cite{KudlaMillson1,KudlaMillson2,KudlaMillson3}. Perhaps more importantly, our results show that the  $\varphi_{KM}$ are characteristic forms.

\subsection{Notation} \begin{itemize}
\item $A^{p,q}(X)$ : complex-valued $(p,q)$-forms on a complex manifold $X$.
\item $\mathcal{F}^\vee$ : dual bundle of a vector bundle $\mathcal{F}$.
\item $\alpha[k]$ : component in $A^k(X)$ of a differential form $\alpha \in A^*(X)$.
\end{itemize}

\subsection{Acknowledgements}
The author is grateful to Stephen Kudla for many helpful conversations and suggestions. This paper has also benefitted from discussions with Daniel Disegni, Daniel Le, Kartik Prasanna and Siddarth Sankaran; the author thanks all of them.

\section{Superconnections and differential forms on period domains}
In this section, we first recall some basic facts about period domains and Hodge bundles (\textsection\ref{subsection:PHS_and_period_domains}-\textsection\ref{subsection:definition_s_v}). Then we construct the superconnections $\nabla_v$ and the forms $\varphi(\cdot)$ (\textsection \ref{subsection:superconnection_nabla_v_def}) and prove some basic properties of $\varphi(\cdot)$ (\textsection \ref{subsection:main_properties_varphi_v}).

Throughout this paper we fix the following:
\begin{itemize}
\item[]$V$ : a $\mathbb{Q}$-vector space of finite dimension $m >0$;
\item[]$w$ : a positive even integer;
\item[]$Q: V \times V \to \mathbb{Q}$ : a non-degenerate, indefinite symmetric bilinear form on $V$, with signature $(s,t)$.
\end{itemize}
We define ${\bf G}=Aut(V,Q)$ to be the orthogonal group of $(V,Q)$ and set
\begin{equation}
G = {\bf G}(\mathbb{R}).
\end{equation}
Thus $G$ is an orthogonal group isomorphic to $O(s,t)$. We denote the connected component of the identity in $G$ by $G^0$.

\subsection{Polarized Hodge structures and period domains} \label{subsection:PHS_and_period_domains} We consider $\mathbb{Q}$-Hodge structures 
\begin{equation}
h: Res_{\mathbb{C}/\mathbb{R}} \mathbb{G}_m \to GL(V_\mathbb{R})
\end{equation}
of weight $w$ on $V_\mathbb{Q}$ that are polarized by $Q$. Writing $(F^p)_{0 \leq p \leq w}$ for the Hodge filtration of $h$, this means that the Riemann bilinear relations
\begin{equation} \label{eq:Riemann_bilinear}
\begin{split}
Q(F^p,F^{w-p+1}) &= 0, \\
Q(Cv, \overline{v}) & > 0 \text{ for } v \text{ non-zero in } V_\mathbb{C},
\end{split}
\end{equation}
hold, where $C=h(i)$. We write $h^{p,q}=\dim_\mathbb{C} V^{p,q}$
for the Hodge numbers of $h$ and call ${\bf h}=(h^{w,0},h^{w-1,1},\ldots,h^{0,w})$ the type of $h$. We only consider types with $h^{w/2,w/2} \neq 0$. The polarized Hodge structures $(V,Q,h)$ of a given type ${\bf h}$ can be identified with the points of a complex manifold $\mathbb{D}_{\bf h}$. Namely, the first bilinear relation defines a closed subvariety
\begin{equation}
\mathbb{D}_{\bf h}^\vee := \{ (F^p)_{0 \leq p \leq w} | Q(F^p,F^{w-p+1})=0 \}
\end{equation}
of the variety of flags in $V_\mathbb{C}$ of type ${\bf h}$. The second Riemannian bilinear relation defines an open subset $\mathbb{D}_{\bf h} \subset \mathbb{D}^\vee_{\bf h}$; thus $\mathbb{D}_{\bf h}$ is the period domain of type ${\bf h}$ and $\mathbb{D}_{\bf h}^\vee$ is its compact dual. If ${\bf h}$ is fixed, we omit it from the notation and simply write $\mathbb{D}$ and $\mathbb{D}^\vee$.

\subsection{Hodge bundles and Hodge loci} \label{subsection:Hodge_bundles_and_Hodge_loci} We write $\mathcal{V}$ for the trivial holomorphic vector bundle on $\mathbb{D}^\vee$ with fiber $V_\mathbb{C}$, and identify it with its sheaf of sections $V_\mathbb{C} \otimes \mathcal{O}_{\mathbb{D}^\vee}$. Through the inclusion $V_\mathbb{R} \subset V_\mathbb{C} \otimes \mathcal{O}_{\mathbb{D}^\vee}$ we can view $v \in V_\mathbb{R}$ as a section of $\mathcal{V}$. For every $p$ with $0 \leq p \leq w$, there is a holomorphic subbundle $\mathcal{F}^p \subset \mathcal{V}$ whose fiber over a point $z=(F^p)_{0 \leq p \leq w}$ is $F^p \subset V_\mathbb{C}$. We denote by the same symbols $\mathcal{F}^p$ and $\mathcal{V}$ the restrictions of these bundles to $\mathbb{D}$. The $\mathcal{F}^p$ are known as Hodge bundles, and on $\mathbb{D}$ they carry a natural hermitian metric $(\cdot,\cdot)_{\mathcal{F}^p}$ defined by
\begin{equation}
(v_z,v'_z)_{\mathcal{F}^p} = 2Q(C v_z,\overline{v'_z}).
\end{equation}
For any $v \in V_\mathbb{R}$, define the Hodge locus
\begin{equation}
\mathbb{D}_v := \{z \in \mathbb{D} | v \in \mathcal{F}_z^{w/2}\}.
\end{equation}
Note that $\mathbb{D}_v$ is non-empty only if $Q(v,v)>0$ or $v=0$. When $v \in V_\mathbb{Q}$, the bilinear relations \eqref{eq:Riemann_bilinear} show that $\mathbb{D}_v$ agrees with the locus of $z \in \mathbb{D}$ where $v$ is a Hodge class.

\subsection{Action of $G$} \label{subsection:Action_of_G} The assignment $(g,\varphi) \mapsto g \circ \varphi$ defines an action of $G$ on $\mathbb{D}$ by holomorphic transformations. This action is transitive and the stabilizer of a point in $\mathbb{D}$ is a compact subgroup of $G$ isomorphic to
\begin{equation}
O(h^{w/2,w/2}) \times \prod_{p>w/2} U(h^{p,w-p}).
\end{equation}
In particular, $\mathbb{D}$ has two connected components; we fix one and denote it by $\mathbb{D}^+$. The Hodge bundles $\mathcal{F}^p$ and their hermitian metrics $(\cdot,\cdot)_{\mathcal{F}^p}$ are naturally $G$-equivariant. 

\subsection{Definition of the section $s_v$} \label{subsection:definition_s_v} Set 
\begin{equation}
\mathcal{F}=\mathcal{F}^{w/2+1}.
\end{equation}
The bilinear form $Q$ gives an identification $\mathcal{V} \cong \mathcal{V}^\vee$. Through it, an element $v \in V_\mathbb{R}$ defines a global section $s_v$ of $\mathcal{F}^\vee$; more concretely, $s_v$ is defined by
\begin{equation} \label{eq:def_s_v}
s_v(v'_z)=Q(v'_z,v).
\end{equation}
There is a unique hermitian metric $(\cdot,\cdot)_{\mathcal{F}^\vee}$ on $\mathcal{F}^\vee$ such that the isomorphism $\overline{\mathcal{F}} \cong \mathcal{F}^\vee$ induced by $(\cdot,\cdot)_{\mathcal{F}}$ is an isometry. For a section $s_z$ of $\mathcal{F}^\vee$, define
\begin{equation}
h_z(s)=(s_z,s_z)_{\mathcal{F}^\vee}.
\end{equation}
Writing $v=\overline{v_z} + v'_z$ with $v_z \in \mathcal{F}_z$ and $v'_z \in (\mathcal{F}_z)^\perp=\mathcal{F}_z^{w/2}$, we have
\begin{equation}
h_z(s_v) = 2Q(C v_z,\overline{v_z}).
\end{equation}
By \eqref{eq:Riemann_bilinear}, the zero set of the section $s_v$ is $\mathbb{D}_v$; this shows that each $\mathbb{D}_v$ is an analytic subset of $\mathbb{D}$, with analytic structure given by the exact sequence
\begin{equation} \label{eq:Koszul_p_coker}
\mathcal{F} \xrightarrow{s_v} \mathcal{O}_\mathbb{D} \to \mathcal{O}_{\mathbb{D}_v} \to 0.
\end{equation}
Assume that $Q(v,v)>0$ and let $G_v$ be the stabilizer of $v$ in $G$. Then $G_v$ is isomorphic to $O(s-1,t)$ and $\mathbb{D}_v$ is a homogeneous complex manifold under $G_v$. The stabilizers of this action are isomorphic to $O(h^{w/2,w/2}-1) \times \Pi_{p>w/2} U(h^{p,w-p})$. We conclude that the complex codimension of $\mathbb{D}_v$ in $\mathbb{D}$ is
\begin{equation}
codim_{\mathbb{D}} \ \mathbb{D}_v = rk(\mathcal{F}^\vee) = h^{w/2+1,w/2-1}+\ldots +h^{w,0}.
\end{equation}
In particular, $s_v$ is a regular section of $\mathcal{F}^\vee$.

\subsection{Superconnections and differential forms on $\mathbb{D}$} \label{subsection:superconnection_nabla_v_def} Define $K(v)$ to be the Koszul complex associated with $s_v$ (see \autoref{appendix}); it carries a hermitian metric induced from the metric on $\mathcal{F}$, with corresponding Chern connection $\nabla$. We regard $K(v)$ as a super vector bundle with even part $\wedge^{even}\mathcal{F}$ and odd part $\wedge^{odd} \mathcal{F}$ and denote by $s_v^*$ the adjoint to $s_v \in End(\wedge \mathcal{F})^{odd}$.
\begin{definition}
Let $\nabla_v$ be the superconnection on $K(v)$ given by
\begin{equation}
\nabla_v=\nabla + i \cdot \sqrt{2\pi} ( s_v + s_v^* )
\end{equation}
and define
\begin{equation}
\varphi^0(v)=tr_s(e^{\nabla_v^2}) \in \oplus_{p \geq 0} A^{p,p}(\mathbb{D}).
\end{equation}

More generally, given vectors $v_1,\ldots,v_r \in V_\mathbb{R}$, we write $K(v_1,\ldots,v_r)$ for the Koszul complex associated with $(s_{v_1},\ldots,s_{v_r}) : \mathcal{F}^{\oplus r} \to \mathcal{O}_{\mathbb{D}}$, $\nabla_{v_1,\ldots,v_r}$ for the corresponding superconnection and
\begin{equation}
\varphi^0(v_1,\ldots,v_r)=tr_s(e^{\nabla_{v_1,\ldots,v_r}^2}) \in \oplus_{p \geq 0} A^{p,p}(\mathbb{D}).
\end{equation}
\end{definition}
%
As we will see below, for fixed $v$ with $Q(v,v)>0$ (so that $\mathbb{D}_v \neq \emptyset$), the form $\varphi^0(v)$ decreases very rapidly as we move away from $\mathbb{D}_v$. However, as a function of $v$ it is not rapidly decreasing: for example, when restricted to $\mathbb{D}_v$, the form $\varphi^0(tv)$ is independent of $t>0$. This makes it impossible to define interesting theta series as sums of forms $\varphi^0(v)$ where $v$ varies in a lattice inside $V_\mathbb{Q}$ since the sum will be divergent over the locus of Hodge classes. Fortunately, one can rescale $\varphi^0(v)$ to obtain a form $\varphi(v)$ with better growth properties as a function of $v$, as follows. 

\begin{definition} For $v_1,\ldots,v_r \in V_\mathbb{R}$, define
\begin{equation*}
\varphi(v_1,\ldots,v_r) = e^{-\pi \sum Q(v_i,v_i)} \varphi^0(v_1,\ldots,v_r) \in \oplus_{p \geq 0} A^{p,p}(\mathbb{D}).
\end{equation*}
\end{definition}

We will show that $\varphi(v)$ decreases rapidly with $v$ in \autoref{subsection:upper_bound_Duhamel}. Note that by \eqref{eq:app_tensor_Chern} and \eqref{eq:app_tensor_Koszul}, we have
\begin{equation} \label{eq:varphi_wedge}
\varphi(v_1,\ldots,v_r)=\varphi(v_1) \wedge \ldots \wedge \varphi(v_r).
\end{equation}

\subsection{Basic properties of the forms $\varphi$} \label{subsection:main_properties_varphi_v} We now list some basic properties of the forms $\varphi(v_1,\ldots,v_r)$. The proofs are straightforward consequences of the properties of the Chern form outlined in \autoref{appendix}. Given a vector bundle with connection $(E,\nabla)$, we write 
\begin{equation}
c^{top}(E,\nabla)=\det(\nabla^2)
\end{equation}
for its Chern-Weil form of top degree and
\begin{equation}
Td(E,\nabla)=\det \left( \frac{\nabla^2}{1-e^{-\nabla^2}} \right)
\end{equation}
for its Todd form.

\begin{proposition} \label{prop:phi_main_properties} Let $r$ be a positive integer and $v_1,\ldots,v_r \in V_\mathbb{R}$.
\begin{enumerate}[label=(\alph*)]
\item $\varphi(v_1,\ldots,v_r)$ is closed.
\item For every $g \in G$, we have $g^*\varphi(gv_1,\ldots, gv_r) = \varphi(v_1,\ldots,v_r)$.
\item $\varphi(v_1,\ldots,v_r)[2k]=0$ if $k < r$.
\item $\varphi(0) = c^{top}(\mathcal{F}^\vee,\nabla) \wedge Td(\mathcal{F}^\vee,\nabla)^{-1}$.
\item If $h \in SO(r)$, then $\varphi((v_1,\ldots,v_r) \cdot h) = \varphi(v_1,\ldots,v_r)$.
\end{enumerate}
\end{proposition}

\begin{proof}
Property $(a)$ holds for general Chern character forms. For any $v \in V_\mathbb{R}$, the $G$-equivariant structure on the Hodge bundles induces an isomorphism of complexes $(g^{-1})^*K(v) \cong K(gv)$ preserving the connection $\nabla$, and this proves $(b)$. For any $z \in \mathbb{D}$, we have 
\begin{equation}
\nabla_v^2[0](z)=-2\pi h_z(s_v) \cdot id \in End(\wedge \mathcal{F}_z),
\end{equation}
hence $\varphi(v)[0]=0$ and $(c)$ follows from \eqref{eq:varphi_wedge}. We have
\begin{equation}
\varphi(0)=ch(\wedge \mathcal{F},\nabla)=\det \left(1-e^{\nabla_\mathcal{F}^2} \right),
\end{equation}
and hence $(d)$ holds. To prove $(e)$, note that $h$ induces an isometry
\begin{equation}
i(h) : K(v_1,\ldots,v_r) \cong K((v_1,\ldots,v_r)\cdot h)
\end{equation}
(see \eqref{eq:Koszul_SO_isom}) such that 
\begin{equation}
\nabla_{(v_1,\ldots,v_r)\cdot h}=i(h)^{-1}\nabla_{v_1,\ldots,v_r} i(h)
\end{equation}
and the result follows since $tr_s$ is invariant under conjugation.
 \end{proof}

Finally, we note that the form $\varphi(v_1,\ldots,v_r)^*$ is invariant under complex conjugation (see \autoref{appendix} for the definition of the operator $^*$). 

\section{An example: the Hermitian symmetric domain attached to $SO(n,2)$} \label{section:example_SO_n_2}

In this section, we focus on a special case where $\mathbb{D}$ is the hermitian symmetric domain associated with $SO(n,2)$. After describing the Hodge bundle $\mathcal{L}$, its Hodge metric (\textsection\ref{subsection:period_domain_herm_sym_case}-\textsection\ref{subsection:taut_bundle_Hodge_metric_and_connection}) and the superconnection $\nabla_v$ (\textsection \ref{subsection:superconnection_def_herm_sym_case}), we compute explicitly the degree $2$ component of $\varphi(v)$ in \textsection \ref{subsection:explicit_comp_deg_2}. We use this to compare $\varphi(v)$ with the Kudla-Millson forms $\varphi_{KM}$ in \textsection \ref{subsec:computation_upper_half_plane}-\textsection \ref{subsection:comparison_varphi_varphi_KM}, where we prove \autoref{thm:comparison_varphi_varphi_KM}.

\subsection{Period domains for polarized Hodge structures of type $(1,n,1)$} \label{subsection:period_domain_herm_sym_case} We now consider the special case where the the bilinear form $Q$ has signature $(n,2)$ with $n \geq 1$. 
Define
\begin{equation}
\begin{split}
\mathcal{P} &=\{ v \in V_\mathbb{C} | Q(v,v)=0, \ Q(v,\overline{v}) <0 \},\\
\mathbb{D} &= \mathcal{P}/\mathbb{C}^\times.
\end{split}
\end{equation}
Since a Hodge structure of weight $2$ and type $(1,n,1)$ on $V$ polarized by $Q$ is determined by the line $F^2 \subset V_\mathbb{C}$, the complex manifold $\mathbb{D}$ is the period domain for polarized Hodge structures on $(V,Q)$ of this type.

%
The group $G$ acts transitively on $\mathbb{D}$ leaving the complex structure invariant, and the quotient map $\mathcal{P} \to \mathbb{D}$ realizes $\mathcal{P}$ as a $G$-equivariant holomorphic principal bundle with fiber $\mathbb{C}^\times$. We write $\mathcal{L} \to \mathbb{D}$ for the associated complex line bundle (so that $\mathcal{L}-\{0\} \cong \mathcal{P}$), known as the tautological bundle. The natural map $\mathbb{D} \to \mathbb{P}(V(\mathbb{C}))$ embeds $\mathbb{D}$ as an open subset of the quadric
\begin{equation}
\{ v \in \mathbb{P}(V(\mathbb{C})) | Q(v,v)=0\} \subset \mathbb{P}(V(\mathbb{C}))
\end{equation} 
and identifies $\mathcal{L}$ with the pullback of $\mathcal{O}_{\mathbb{P}(V(\mathbb{C}))}(-1)$. The linear functional $Q(\cdot,v)$ on $V(\mathbb{C})$ attached to $v \in V_\mathbb{R}$ defines a global section of $\mathcal{O}_{\mathbb{P}(V(\mathbb{C}))}(1)$ and hence by restriction a global section of $\mathcal{L}^\vee$ that we denote by $s_v$. Note that $\mathcal{L} \cong \mathcal{F}^2$ as $G$-equivariant line bundles and $s_v$ is as defined in \autoref{subsection:definition_s_v}.

\subsection{Hodge metric and Chern connection on $\mathcal{L}$} \label{subsection:taut_bundle_Hodge_metric_and_connection} The bilinear form on $V(\mathbb{C})$ induces a hermitian metric $h_\mathcal{L}$ on $\mathcal{L}$: for $z=[v] \in \mathbb{D}$, we set
\begin{equation}
h_{\mathcal{L}_z}(v)= -2Q(v,\overline{v}).
\end{equation}
There is a unique hermitian metric on $\mathcal{L}^\vee$ making the isomorphism $\overline{\mathcal{L}} \cong \mathcal{L}^\vee$ induced by $h_\mathcal{L}$ an isometry; we denote this metric by $h$. Writing $v_z$ for the orthogonal projection of $v$ to $(\mathcal{L}_z \oplus \overline{\mathcal{L}_z}) \cap V_\mathbb{R}$, we have
\begin{equation}
h_z(s_v) = -Q(v_z,\overline{v_z}).
\end{equation}

We endow $\mathcal{O}_\mathbb{D}$ with the metric $|\cdot|$ and denote by $s^*_v: \mathcal{O}_\mathbb{D} \to \mathcal{L}$ the adjoint of $s_v:\mathcal{L} \to \mathcal{O}_\mathbb{D}$ with respect to these metrics, so that
\begin{equation} \label{eq:h(s_v)}
s_v^*s_v(z) = s_vs_v^*(z)=h_z(s_v).
\end{equation}
We denote by $\nabla_{\mathcal{L}}$ the Chern connection on $(\mathcal{L},h_\mathcal{L})$. On $\mathbb{D}-\mathbb{D}_v$, the section $s_v^{-1}$ gives a trivialization of $\mathcal{L}$. We have
\begin{equation}
\nabla_{\mathcal{L}}(fs_v^{-1}) = df \otimes s_v^{-1}- f\frac{\partial h(s_v)}{h(s_v)} \otimes s_v^{-1}
\end{equation}
and hence the curvature of $\mathcal{L}$ is given by
\begin{equation} \label{eq:curvature_L_dual}
\nabla_{\mathcal{L}}^2 = d(-\frac{\partial h(s_v)}{h(s_v)})=\frac{\overline{\partial}h(s_v) \wedge \partial h(s_v)}{h(s_v)^2}-\frac{\overline{\partial}\partial h(s_v)}{h(s_v)}.
\end{equation}

\subsection{The superconnection $\nabla_v$ and the form $\varphi(v)$} \label{subsection:superconnection_def_herm_sym_case} We regard $\mathcal{O}_\mathbb{D} \oplus \mathcal{L}$ as a super line bundle, with even part $\mathcal{O}_\mathbb{D}$ and odd part $\mathcal{L}$. To a vector $v \in V_\mathbb{R}$ we attach the superconnection
\begin{equation}
\nabla_v = \nabla_{\mathcal{O}_\mathbb{D}} + \nabla_{\mathcal{L}} + i \cdot \sqrt{2\pi} \left( \begin{array}{cc} 0 & s_v \\ s_v^* & 0 \end{array} \right),
\end{equation}
where $\nabla_{\mathcal{O}_\mathbb{D}}:=d$. Define
\begin{equation}
\varphi(v) = e^{-\pi Q(v,v)} \cdot tr_s(e^{\nabla_v^2}) \in \oplus_{p \geq 0} A^{p,p}(\mathbb{D}).
\end{equation}
More generally, given vectors $v_1,\ldots,v_r \in V_\mathbb{R}$, we consider the total complex of
\begin{equation}
\otimes_{1 \leq i \leq r} (s_{v_i}: \mathcal{L} \to \mathcal{O}_\mathbb{D})
\end{equation}
as a super vector bundle, and denote by $\nabla_{v_1,\ldots,v_r}$ the superconnection induced by the $\nabla_{v_i}$. We define
\begin{equation}
\varphi(v_1,\ldots,v_r) = e^{-\pi (\sum_i Q(v_i,v_i))} \cdot tr_s(e^{\nabla_{v_1,\ldots,v_r}^2}).
\end{equation}
By \eqref{eq:Chern_product}, we have $\varphi(v_1,\ldots,v_r) = \varphi(v_1) \wedge \cdots \wedge \varphi(v_r)$.

\subsection{Computation of $\varphi(v)[2]$} \label{subsection:explicit_comp_deg_2} We now give an explicit formula for the degree two component of $\varphi(v)$. Note that $\mathcal{O}_\mathbb{D} \oplus \mathcal{L}$ is a super line bundle; in particular, the algebra $\Gamma(End(\mathcal{O}_\mathbb{D} \oplus \mathcal{L}))$ is supercommutative. As the tensor product of two supercommutative algebras, the algebra
\begin{equation}
A(\mathbb{D},End(\mathcal{O}_\mathbb{D} \oplus \mathcal{L}))=A^*(\mathbb{D}) \hat{\otimes}_{\mathcal{C}^\infty(X)} \Gamma(End(\mathcal{O}_\mathbb{D} \oplus \mathcal{L})),
\end{equation}
is also supercommutative. Writing $\nabla_v^2=r_2+r_1+r_0$ with $r_i \in A^i(\mathbb{D},End(\mathcal{O}_\mathbb{D} \oplus \mathcal{L}))$ even, we see that $r_0$, $r_1$ and $r_2$ commute. We conclude that
\begin{equation}
\begin{split}
tr_s(e^{\nabla_v^2})[2] &= tr_s(e^{r_0} e^{r_1} e^{r_2})[2] \\
&=e^{r_0}\left(tr_s(r_2)+\frac{1}{2}tr_s(r_1^2) \right)
\end{split}
\end{equation}
since
\begin{equation}
r_0=-2\pi \left( \begin{array}{cc} s_vs_v^* & 0 \\ 0 & s_v^*s_v \end{array} \right) \in \Gamma(End(\mathcal{O}_\mathbb{D} \oplus \mathcal{L}))
\end{equation}
is a scalar. It remains to compute $tr_s(r_2)$ and $tr_s(r_1^2)$. By \eqref{eq:curvature_L_dual}, we have
\begin{equation}
tr_s(r_2)=tr_s(\nabla_{\mathcal{L}}^2)=- \left(\frac{\overline{\partial}h(s_v) \wedge \partial h(s_v)}{h(s_v)^2}-\frac{\overline{\partial}\partial h(s_v)}{h(s_v)}
\right).
\end{equation}
Consider now $r_1^2$. We have
\begin{equation}
\begin{split}
(\nabla_{\mathcal{O}_\mathbb{D}}s_v+s_v\nabla_{\mathcal{L}})(s_v^{-1}) &= s_v \nabla_{\mathcal{L}} s_v^{-1} = -\frac{\partial h(s_v)}{h(s_v)} \otimes s_v s_v^{-1} = -\frac{\partial h(s_v)}{h(s_v)} \otimes 1\\
(\nabla_{\mathcal{L}} s_v^* + s_v^*\nabla_{\mathcal{O}_\mathbb{D}})(1) &= \nabla_{\mathcal{L}}(h(s_v) s_v^{-1}) \\ 
&= (dh(s_v)-\frac{\partial h(s_v)}{h(s_v)}h(s_v)) \otimes s_v^{-1} = \overline{\partial}h(s_v) \otimes s_v^{-1}
\end{split}
\end{equation}
and hence
\begin{equation}
\frac{1}{2}tr_s(r_1^2) = 2\pi \frac{\partial h(s_v) \wedge \overline{\partial}h(s_v)}{h(s_v)}.
\end{equation}
By \eqref{eq:h(s_v)}, we arrive at the formula
\begin{equation}\label{eq:explicit_deg_2}
\varphi(v)[2]=e^{-\pi(Q(v,v)+2h(s_v))}\left(-\frac{\overline{\partial}h(s_v) \wedge \partial h(s_v)}{h(s_v)^2}+\frac{\overline{\partial}\partial h(s_v)}{h(s_v)} + 2\pi \frac{\partial h(s_v) \wedge \overline{\partial}h(s_v)}{h(s_v)} \right).
\end{equation}

\subsection{The case $n=1$} \label{subsec:computation_upper_half_plane} Consider the case $n=1$, where $\mathbb{D}$ can be identified with the Poincar\'e upper half plane $\mathbb{H}$. In \cite[p. 603]{KudlaCD}, a form
\begin{equation}\varphi_{KM}(v) \in A^{1,1}(\mathbb{H})
\end{equation}
is defined for every $v \in V_\mathbb{R}$. With the notation of that paper, for $\tau \in \mathbb{H}$ we have $R(v,\tau)=h_\tau(s_v)$ (ibid., (11.14)) and comparing \eqref{eq:explicit_deg_2} with ibid., (11.27), (11.37)-(11.40) shows that
\begin{equation}
\varphi(v)^*[2] = \varphi_{KM}(v), \quad v \in V_\mathbb{R}.
\end{equation}

\subsection{Restriction of $\varphi(v)$ to a hermitian subdomain of type $SO(n-1,2)$} Suppose that $w \in V_\mathbb{R}$ satisfies $Q(w,w)>0$; then $\mathbb{D}_w:=\mathrm{div}(s_w)$ is non-empty. The stabiliser $G_w$ of $w$ in $G$ is isomorphic to $SO(n-1,2)$ and $\mathbb{D}_w$ can be identified with the hermitian symmetric domain attached to $G_w$; denote by $\mathcal{L}_w$ the corresponding tautological bundle. The restriction of $\mathcal{L}$ to $\mathbb{D}_w$ is then isometric to $\mathcal{L}_w$, and for any $v \in V_\mathbb{R}$ we have an isometry
\begin{equation}
\left. (\mathcal{L} \overset{s_v}{\to} \mathcal{O}_\mathbb{D}) \right|_{\mathbb{D}_w} \cong \mathcal{L}_w \overset{s_{v'}}{\to}
 \mathcal{O}_{\mathbb{D}_w}
\end{equation} 
with $v'$ the orthogonal projection of $v$ to $w^\perp$; this isometry identifies $\nabla_v|_{\mathbb{D}_w}$ with $\nabla_{v'}$. We obtain the following formula for the restriction of $\varphi(v)$ to $\mathbb{D}_w$:
\begin{equation} \label{eq:formula_restriction}
\left. \varphi(v) \right|_{\mathbb{D}_w} = e^{-\pi Q(v'',v'')} \cdot \varphi(v'), \quad v'':=v-v'.
\end{equation}

\subsection{Comparison with the Kudla-Millson forms} \label{subsection:comparison_varphi_varphi_KM} Using formula \eqref{eq:formula_restriction} for the restriction and the comparison with $\varphi_{KM}$ in \autoref{subsec:computation_upper_half_plane}, we can now show how to recover the forms $\varphi_{KM}(v_1,\ldots,v_r)$ in \cite[Thm. 7.1]{KudlaOrthogonal} (denoted there by $\varphi^{(r)}$) from the forms $\varphi(v_1,\ldots,v_r)$. We note that 
\begin{equation} \label{eq:wedge_varphi_KM}
\varphi_{KM}(v_1,\ldots,v_r) = \varphi_{KM}(v_1) \wedge \ldots \wedge \varphi_{KM}(v_r).
\end{equation}
\begin{theorem} \label{thm:comparison_varphi_varphi_KM}
For any $v_1,\ldots,v_r \in V_\mathbb{R}$ ($r \geq 1$), we have
\begin{equation}
\varphi(v_1,\ldots,v_r)^*[2r] = \varphi_{KM}(v_1,\ldots,v_r).
\end{equation}
\end{theorem}
\begin{proof}
By Proposition~\ref{prop:phi_main_properties}.$(c)$ and \eqref{eq:wedge_varphi_KM} we may assume that $r=1$. Argue by induction on $n$, where the case $n=1$ is \autoref{subsec:computation_upper_half_plane}. So let $n \geq 2$ and assume that the statement holds for $n-1$. Consider an analytic divisor $\mathbb{D}_w=\mathrm{div}(s_w)$ in $\mathbb{D}$. The restriction to $\mathbb{D}_w$ of the form $\varphi_{KM}$ is described in \cite[Lemma 7.3]{KudlaOrthogonal}. Comparing with \eqref{eq:formula_restriction}, we see that the restrictions of $\varphi(v)^*[2]$ and $\varphi_{KM}(v)$ to any such divisor agree. This is enough since the forms $\varphi(v)^*[2]$ and $\varphi_{KM}(v)$ are real and the global sections of $\mathcal{O}_{\mathbb{P}(V(\mathbb{C}))}(1)$ separate points and tangent vectors. 
 \end{proof}

\section{Properties of $\varphi$} \label{section:Properties_of_varphi}

We now go back to the general case and study the forms $\varphi(\cdot)$ on an arbitrary period domain $\mathbb{D}$. After showing that $\varphi(v)$ is smooth and rapidly decreasing in $v$ in \textsection \ref{subsection:upper_bound_Duhamel}, we recall the formulas defining the Weil representation $\omega$ of the metaplectic double cover $Mp_{2r,\mathbb{R}}$ on $\mathcal{S}(V_\mathbb{R}^r)$ in \textsection \ref{subsection:Weil_representation}. Then (\textsection \ref{subsection:behaviour_varphi_max_compact}-\textsection \ref{subsection:behaviour_varphi_lowering}) we determine the behaviour of $\varphi(\cdot)$ under certain operators defined using $\omega$, and we prove \autoref{thm:thm_1}. We conclude by showing (\textsection \ref{subsection:Thom_property}) that the forms $\varphi(v_1,\ldots,v_r)$ have a Thom form property on certain arithmetic quotients of $\mathbb{D}^+$. The proof uses a result of Bismut \cite{BismutInv90} that, in turn, relies on the Mathai-Quillen formula computing the Chern form of Koszul complexes.

\subsection{Rapid decay of $\varphi(\cdot)$} \label{subsection:upper_bound_Duhamel} Let us give an upper bound on the growth of $\varphi(v)$ as a function of $v$. Following \cite[p. 144]{SouleBook}, we use Duhamel's formula: this states that for any two endomorphisms $A$, $B$ of a finite-dimensional complex vector space, we have
\begin{equation}
e^{-B}-e^{-A} = - \int_0^1 e^{-sA}(B-A)e^{-(1-s)B}ds.
\end{equation}
(Proof: the integrand is the derivative of $e^{-sA}e^{-(1-s)B}$.) Applying this to the elements $\nabla_v^2(z)$ and $\nabla_v^2[0](z)$ of $End(\wedge T^*\mathbb{D}_z \otimes \wedge \mathcal{F}_z)$ and iterating, we can rewrite $\varphi$ as follows.
Let
\begin{equation} \label{eq:def_q_z}
q_z(v):=\frac{1}{2}Q(v,v)+h_z(s_v)
\end{equation}
and note that, by \eqref{eq:Riemann_bilinear}, the quadratic form $q_z : V_\mathbb{R} \to \mathbb{R}$ is positive definite for any $z \in \mathbb{D}$. Let 
\begin{equation}
\Delta^k=\{(t_1,\ldots,t_k) \in \mathbb{R}^k | 0 \leq t_1 \leq \ldots \leq t_k \leq 1 \}
\end{equation}
be the $k$-simplex. Setting $\nabla_v^2(z)=\nabla_v^2[0](z)+R(v,z)$, we have
\begin{equation}
\nabla_v^2[0](z)=-2\pi(s_vs_v^* + s_v^*s_v)(z)=-2\pi h_z(s_v) \cdot id,
\end{equation}
and hence
\begin{equation} \label{eq:Duhamel_expansion}
\begin{split}
&e^{-\pi Q(v,v)}   e^{\nabla_v^2(z)} = e^{-2\pi q_z(v)} + \\
& \sum_{k \geq 1} (-1)^k \int_{\Delta^k} e^{-(1-t_k)2\pi q_z(v)}R(v,z) e^{-(t_k-t_{k-1})2\pi q_z(v)} \cdots R(v,z) e^{-t_1 2\pi q_z(v)}dt_1 \cdots dt_k,
\end{split}
\end{equation}
where the sum has at most $n=\dim_\mathbb{R} \mathbb{D}$ terms since $R(v,z)$ has positive degree. 

Let $||\cdot||_{K,k}$ be an algebra seminorm on the algebra of endomorphisms of the vector bundle $\wedge T^*\mathbb{D} \otimes \wedge \mathcal{F}$ measuring uniform convergence on a compact subset $K$ of partial derivatives up to order $k$. Let $q_K:V_\mathbb{R} \to \mathbb{R}$ be a positive definite quadratic form such that $q_K \leq q_z$ for every $z \in K$. Arguing as in \cite[\textsection 4]{QuillenChern} one shows that there is a constant $M=M_{K,k}$ such that
\begin{equation}
||e^{-2\pi q_z(v)}||_{K,k} \leq M e^{-2\pi q_K(v)}, \quad v \in V_\mathbb{R}.
\end{equation}
Since $R(v,z) = \nabla^2 + [\nabla,i \sqrt{2 \pi}(s_v + s_v^*)]$ grows linearly with $v$, we can find an affine function on $V_\mathbb{R}$ giving an upper bound for $||R(v,z)||_{K,k}$. Using \eqref{eq:Duhamel_expansion}, we conclude that there are constants $C=C_{K,k}$ and $a$ such that
\begin{equation} \label{eq:varphi_estimate}
||e^{-\pi Q(v,v)} e^{\nabla_v^2(z)}||_{K,k} \leq C(1+q_K(v))^a e^{-2\pi q_K(v)}, \quad v \in V_\mathbb{R}.
\end{equation}
The same argument shows that a similar bound holds after taking any number of derivatives with respect to $v$. 

We denote by $\mathcal{S}(V_\mathbb{R}^r)$ the Schwartz space of smooth, rapidly decreasing functions on $V_\mathbb{R}^r$. The estimate above implies the following result. 

\begin{proposition} \label{prop:superconn_Schwartz}
For every $z \in \mathbb{D}$, we have $\varphi(v_1,\ldots,v_r,z) \in \mathcal{S}(V_\mathbb{R}^r) \otimes \wedge T^* \mathbb{D}_z$.
\end{proposition}

\subsection{Weil representation} \label{subsection:Weil_representation} As shown in Proposition~\ref{prop:superconn_Schwartz}, the form $\varphi(v_1,\ldots,v_r)$ can be regarded as a differential form on $\mathbb{D}$ valued in the Schwartz space $\mathcal{S}(V_\mathbb{R}^r)$. Let 
\begin{equation}
Sp_{2r}(\mathbb{R}) = \left\{ g \in GL_{2r}(\mathbb{R}) \left| g \left( \begin{array}{cc} 0 & 1_r \\ -1_r & 0 \end{array} \right) {^t g} = \left( \begin{array}{cc} 0 & 1_r \\ -1_r & 0 \end{array} \right) \right. \right\}
\end{equation}
be the symplectic group of rank $r$ and let $Mp_{2r,\mathbb{R}}$ be its metaplectic double cover. The group $Mp_{2r,\mathbb{R}}$ acts on $\mathcal{S}(V_\mathbb{R}^r)$ via the Weil representation $\omega=\omega_\psi$ attached to the additive character $\psi(x):=e^{2\pi ix}$. Let us recall the formulas that define this representation.

Let $(\cdot,\cdot)_\mathbb{R}$ be the Hilbert symbol on $\mathbb{R}$ and $h_\mathbb{R}(V)$ be the Hasse invariant of $V_\mathbb{R}$. Define a quadratic character $\chi_{V_\mathbb{R}}$ of $\mathbb{R}^\times$ by
\begin{equation}
\chi_{V_\mathbb{R}}(a) = ((-1)^{m(m-1)/2}\det(V_\mathbb{R}),a)_\mathbb{R}.
\end{equation}
Let $\gamma_\mathbb{R}(a\psi)=e^{\tfrac{2\pi i}{8} sgn(a)}$ and $\gamma_\mathbb{R}(a,\psi)=\tfrac{\gamma_\mathbb{R}(a\psi)}{\gamma_\mathbb{R}(\psi)}$. Define an $8$-th root of unity $\gamma_{V_\mathbb{R}}$ by
\begin{equation}
\gamma_{V_\mathbb{R}} = \gamma_\mathbb{R}(\det(V),\tfrac{1}{2}\psi)\gamma_\mathbb{R}(\tfrac{1}{2}\psi)^m h_\mathbb{R}(V).
\end{equation} 
We identify $Mp_{2r,\mathbb{R}}$ with $Sp_{2r}(\mathbb{R}) \times \mu_2$ as in \cite{KudlaSplitting} and consider the subgroups
\begin{equation}
\begin{split}
M &= \left\{ (m(a),\pm 1) \right\}, \\
N &= \left\{ (n(b),1) \right\},
\end{split}
\end{equation}
where for $a \in GL_n(\mathbb{R})$ and $b \in Sym_n(\mathbb{R})$ we write
\begin{equation}
\begin{split}
m(a)&=\left( \begin{array}{cc} a & \\ & ^t a^{-1} \end{array} \right), \ a \in GL_r(\mathbb{R}),\\
n(b) &= \left( \begin{array}{cc} 1_r & b \\ & 1_r \end{array} \right), \ b \in Sym_r(\mathbb{R}).
\end{split}
\end{equation}
The group $P=MN$ is then a maximal parabolic of $Mp_{2r,\mathbb{R}}$, and we have the formulas
\begin{equation}
\begin{split}
\omega(m(a),\epsilon)\varphi(v) &= \epsilon \chi_{V_\mathbb{R}}(\det(a)) \left|\det(a)\right|^{m/2} \varphi(v\cdot a), \\
&\quad \times \left\{ \begin{array}{cc} 1, & \text{ if } m \text{ is even,} \\ \gamma_\mathbb{R}(\det(a),\tfrac{1}{2}\psi)^{-1}, & \text{ if } m \text{ is odd,} \end{array} \right. \\
\omega(n(b),1)\varphi(v) &= \psi(tr(bQ(v,v)/2)) \varphi(v), \\
\omega\left(\left(\begin{array}{cc} & -1_r \\ 1_r &  \end{array}\right),1 \right)\varphi(v) &= \gamma_{V_\mathbb{R}}^{-r} \cdot \int_{V_\mathbb{R}^r} \varphi(w) \psi(-tr Q(v,w))dw,
\end{split}
\end{equation}
where $\varphi \in \mathcal{S}(V_\mathbb{R}^r)$ and $dw$ is the self-dual Haar measure on $V_\mathbb{R}^r$ with respect to the pairing given by $\psi(tr Q(v,w))$.

\subsection{Behaviour of $\varphi$ under the action of the maximal compact subgroup of $Mp_{2r,\mathbb{R}}$} \label{subsection:behaviour_varphi_max_compact} Consider the maximal compact subgroup
\begin{equation}
\left\{ \left. \left( \begin{array}{cc} a & -b \\ b & a \end{array} \right) \right| a+ib \in U(r)  \right\} \cong U(r)
\end{equation}
of $Sp_{2r}(\mathbb{R})$ and denote by $\widetilde{U(r)}$ its inverse image in $Mp_{2r,\mathbb{R}}$. It admits a character 
\begin{equation}
{\det}^{1/2}: \widetilde{U(r)} \to S^1 \subset \mathbb{C}^\times
\end{equation}
whose square factors through $U(r)$ and defines the usual determinant character $\det:U(r) \to \mathbb{C}^\times$. 
Our next goal is to study the behaviour of the forms $\varphi(v_1,\ldots,v_r)$ under the action of $\mathfrak{u}(r)=Lie(U(r))$. For this purpose, it is convenient to fix an orthogonal basis $v_1,\ldots,v_{s+t}$ of $V_\mathbb{R}$ with
\begin{equation}
Q(v_j,v_j) = \left\{ \begin{array}{rl} 1, & \text{ if } 1 \leq j \leq s, \\ -1, & \text{ if } s+1 \leq j \leq t. \end{array} \right.
\end{equation}
We denote the corresponding dual basis of $V_\mathbb{R}^\vee$ by $x_1,\ldots,x_{s+t}$. We write $\chi:\mathfrak{u}(r) \to i\mathbb{R}$ for the character obtained by differentiating $\det:U(r) \to S^1$. When $r=1$, the element $\chi^{-1}(i) \in \mathfrak{u}(1)$ acts on $\mathcal{S}(V_\mathbb{R})$ by the operator
\begin{equation}
\pi i \sum_{1 \leq j \leq s+t} Q(v_j,v_j) \cdot \left(x_j^2- \frac{1}{(2\pi)^2} \frac{d^2}{dx_j^2}\right).
\end{equation} 

\begin{lemma} \label{lemma:weight}
Let $X \in \mathfrak{u}(r)$. There exists a form $\nu(X,\cdot)$ of odd degree on $\mathbb{D}$, valued in $\mathcal{S}(V_\mathbb{R}^r)$, such that
\begin{equation} \label{eq:weight_phi}
\omega(X)\varphi(v) = \frac{m}{2}  \chi(X) \varphi(v) + d\nu(X,v), \quad v \in V_\mathbb{R}^r,
\end{equation}
and satisfying $g^*\nu(X,gv)=\nu(X,v)$ for $g \in G$.
\end{lemma}
\begin{proof}
We first show that it suffices to prove the statement when $r=1$. We identify $\mathfrak{u}(r)$ with the space of skew-hermitian complex $r \times r$ matrices. For $x_1,\ldots,x_r \in \mathbb{R}$, let
\begin{equation}
d(x) = \left( \begin{array}{ccc} x_1 & \\ & \ddots & \\ & & x_r \end{array} \right)
\end{equation}
and consider the Cartan subalgebra $\mathfrak{h} \cong \mathfrak{u}(1)^r$ of $\mathfrak{u}(r)$ defined by
\begin{equation}
\mathfrak{h} = \left\{\left( \begin{array}{cc} 0 & id(x) \\ -id(x) & 0 \end{array} \right)\right\}.
\end{equation}
Let $X=a+ib \in \mathfrak{u}(r)$; then $a \in \mathfrak{so}(r)$ and hence $\omega(a)\varphi = 0$ by Proposition~\ref{prop:phi_main_properties}.(e). Moreover, $b$ is symmetric and hence we can find $k \in SO(r)$ and $h \in \mathfrak{h}$ such that $ib = kh{^tk}$. Thus it suffices to prove \eqref{eq:weight_phi} for $X \in \mathfrak{h}$, and this reduces to the case $r=1$ by \eqref{eq:varphi_wedge}.

Now assume that $r=1$ and let $X=\chi^{-1}(i)$. A direct computation shows that
\begin{equation}
\omega(X)e^{-\pi Q(v,v)} = \frac{im}{2} e^{-\pi Q(v,v)}.
\end{equation}
Hence we find that
\begin{equation}
\omega(X)\varphi(v) = \frac{im}{2} \varphi(v) -\frac{i}{4\pi} e^{-\pi Q(v,v)} \cdot \sum_{1 \leq j \leq s+t} \xi_j(v),
\end{equation}
where
\begin{equation}
\xi_j(v)= 4\pi x_j(v) \frac{d}{dx_j} \varphi^0(v) + Q(v_j,v_j) \frac{d^2}{dx_j^2} \varphi^0(v).
\end{equation}
The transgression formula \eqref{eq:transg_formula} shows that
\begin{equation}
\frac{d}{dx_j} \varphi^0(v) = d \ tr_s(\frac{d \nabla_v}{dx_j} e^{\nabla_v^2})
\end{equation}
and hence
\begin{equation}
\frac{d^2}{dx_j^2} \varphi^0(v) = d \ tr_s(\frac{d \nabla_v}{dx_j} \frac{d}{dx_j} e^{\nabla_v^2})
\end{equation}
and this implies that each $\xi_j$ is exact. Since $\tfrac{d\nabla_v}{dx_j}=i\sqrt{2\pi}(s_{v_j}+s_{v_j}^*)$, we find that \eqref{eq:weight_phi} holds with
$\nu(X,v)=\alpha(v)+\beta(v)$, where
\begin{equation}
\begin{split}
\alpha(v)&= \sqrt{2\pi} e^{-\pi Q(v,v)} \cdot tr_s((s_v+s_v^*)e^{\nabla_v^2}), \\
\beta(v) &= -\frac{i}{4\pi} e^{-\pi Q(v,v)} \sum_{1 \leq j \leq s+t} Q(v_j,v_j) \cdot tr_s(\frac{d\nabla_v}{dx_j} \frac{d}{dx_j}e^{\nabla_v^2}).
\end{split}
\end{equation}
The form $\alpha(v)$ satisfies $g^*\alpha(gv)=\alpha(v)$ for every $g \in G$. Regarding $\beta(v)$, we have $\tfrac{d\nabla_v^2}{dx_j}=[\nabla_v,\tfrac{d\nabla_v}{dx_j}]$ (supercommutator) and hence 
\begin{equation}
\beta(v)=\frac{i}{2}e^{-\pi Q(v,v)} \sum_{k,l \geq 0} \tfrac{1}{(k+l)!}\beta(v)_{k,l},
\end{equation}
where $\beta(v)_{0,0}=0$ and
\begin{equation}
\beta(v)_{k,l} = \sum_{1 \leq j \leq s+t} Q(v_j,v_j) \ tr_s((s_{v_j}+s_{v_j}^*)\nabla_v^{2k} [\nabla_v,s_{v_j}+s_{v_j}^*]\nabla_v^{2l}).
\end{equation}
For $g \in G$, we have $g^*s_{v_j}=s_{g^{-1}v_j}$ and $Q(g^{-1}v_j,g^{-1}v_j)=Q(v_j,v_j)$ and hence
\begin{equation}
\begin{split}
g^*\beta(gv)_{k,l} &= \sum_{1 \leq j \leq s+t} Q(v_j,v_j) \ tr_s((g^*s_{v_j}+g^*s_{v_j}^*)\nabla_v^{2k} [\nabla_v,g^*s_{v_j}+g^*s_{v_j}^*]\nabla_v^{2l}) \\
&= \sum_{1 \leq j \leq s+t} Q(g^{-1}v_j,g^{-1}v_j) \ tr_s((s_{g^{-1}v_j}+s_{g^{-1}v_j}^*)\nabla_v^{2k} [\nabla_v,s_{g^{-1}v_j}+s_{g^{-1}v_j}^*]\nabla_v^{2l}) \\
&= \beta(v)_{k,l},
\end{split}
\end{equation}
and this shows that $g^*\beta(gv)=\beta(v)$.

Finally, the fact that the form $\nu(X,v)$ takes values in $\mathcal{S}(V_\mathbb{R})$ follows from the estimate \eqref{eq:varphi_estimate}, which is valid for derivatives.
 \end{proof}

We remark that it is possible to prove a stronger lemma: indeed, for any $X \in \mathfrak{u}(r)$, we have 
\begin{equation} \label{eq:strong_weight_behaviour}
\omega(X)\varphi(v) = \frac{m}{2}  \chi(X) \varphi(v) + \partial \overline{\partial} \tilde{\nu}(X,v),
\end{equation}
where $\tilde{\nu}(X,\cdot)$ is a sum of $(p,p)$-forms ($p \geq 0$) valued in $\mathcal{S}(V_\mathbb{R}^r)$ and satisfying $g^*\tilde{\nu}(X,gv)=\tilde{\nu}(X,v)$. The proof is similar, but replaces \eqref{eq:transg_formula} with a double transgression formula such as the one discussed by Faltings \cite[pp. 62-63]{FaltingsLecturesARR}.

\subsection{Behaviour of $\varphi(v)$ under the lowering operator of $\mathfrak{sl}_{2,\mathbb{C}}$} \label{subsection:behaviour_varphi_lowering}

Recall that every $g \in Mp_{2,\mathbb{R}}$ admits a unique expression $g=n(x)m(y)\tilde{k}_\theta$, where we write $n(x)=\left( \begin{smallmatrix} 1 & x \\ 0 & 1 \end{smallmatrix} \right)$ for $x \in \mathbb{R}$; $m(y)=\left( \begin{smallmatrix} y^{1/2} &  \\  & y^{-1/2} \end{smallmatrix} \right)$ for $y \in \mathbb{R}_{>0}$; and $\tilde{k}_\theta \in \widetilde{U(1)} \cong \mathbb{R}/4\pi \mathbb{Z}$ for $\theta \in [0,4\pi[$. We think of $(x,y,\theta)$ as coordinates on $Mp_{2,\mathbb{R}}$. Let
\begin{equation}
X_- = \tfrac{1}{2} \left( \begin{smallmatrix} 1 & -i \\ -i & -1 \end{smallmatrix}\right) \in \mathfrak{sl}_{2,\mathbb{C}}.
\end{equation}
Thus $X_-$ is the usual lowering operator of $\mathfrak{sl}_{2,\mathbb{C}}$, whose action on $f \in \mathcal{C}^\infty(Mp_{2,\mathbb{R}})$ under the right regular representation is given by
\begin{equation}
X_- \cdot f = -ie^{-2i \theta} \left(2y \frac{d}{d\overline{\tau}} - \frac{1}{2} \frac{\partial}{\partial \theta} \right) f,
\end{equation}
where $\tfrac{d}{d\overline{\tau}}=\tfrac{1}{2}(\tfrac{d}{dx}+i\tfrac{d}{dy})$. To determine the behaviour of the form $\varphi(\cdot)$ under $X_-$, we will use the following double transgression formula. Let $N \in End(\wedge \mathcal{F})$ the operator that acts on $\wedge^k \mathcal{F}$ by multiplication by $-k$.
For $v \in V_\mathbb{R}$, define
\begin{equation} \label{eq:varphi'_def}
\psi^0(v) = tr_s(N e^{\nabla_v^2}) \in \oplus_{p \geq 0} A^{p,p}(\mathbb{D}).
\end{equation}
Then we have (see \cite[Thm. 1.15]{BismutGilletSoule1})
\begin{equation} \label{eq:double_transgression}
-\frac{1}{t} \partial \overline{\partial}\psi^0(t^{1/2}v) = \frac{d}{dt} \varphi^0(t^{1/2}v), \quad t>0.
\end{equation}
Let 
\begin{equation}
\psi(v)=e^{-\pi Q(v,v)} \psi^0(v).
\end{equation}
Using the estimate \eqref{eq:varphi_estimate}, we find that $\psi(\cdot)$ takes values in $\mathcal{S}(V_\mathbb{R})$. Note also that we have $g^*\psi(gv)=\psi(v)$ for every $g \in G$.

\begin{lemma} \label{lemma:lowering_varphi}
Let $\tau=x+iy \in \mathbb{H}$ and write $g_\tau=\left(\left( \begin{smallmatrix} y^{1/2} & xy^{-1/2} \\ & y^{-1/2} \end{smallmatrix} \right),1\right) \in Mp_{2,\mathbb{R}}$. Then
\begin{equation}
-2i y^2 \frac{d}{d\overline{\tau}} (y^{-m/4} \cdot \omega(g_\tau)\varphi(v))= -\partial \overline{\partial} (y^{-(m/4-1)} \omega(g_\tau)\psi(v)).
\end{equation}
\end{lemma}
\begin{proof}
Writing $Q(v)=\tfrac{1}{2}Q(v,v)$, we have
\begin{equation}
y^{-m/4} \cdot \omega(g_\tau)\varphi(v) = e^{2\pi i Q(v) \tau} \cdot \varphi^0(y^{1/2}v)
\end{equation}
and hence, by \eqref{eq:double_transgression},
\begin{equation} \label{eq:lowering_op}
\begin{split}
-2i y^2 \frac{d}{d\overline{\tau}} (y^{-m/4} \cdot \omega(g_\tau)\varphi(v)) &= -2iy^2 e^{2\pi i Q(v) \tau} \frac{d}{d\overline{\tau}} \varphi^0(y^{1/2}v) \\
&=y^2 e^{2\pi i Q(v) \tau} \cdot \frac{d}{dy} \varphi^0(y^{1/2}v) \\
&= -y e^{2\pi i Q(v) \tau} \cdot \partial \overline{\partial} \psi^0(y^{1/2}v) \\
&= -\partial \overline{\partial} (y^{-(m/4-1)} \omega(g_\tau)\psi(v)).
\end{split}
\end{equation}
 
\end{proof}
Together with Proposition~\ref{prop:superconn_Schwartz} and Lemma~\ref{lemma:weight}, this proves \autoref{thm:thm_1}.

\begin{remark}
In the setting of \autoref{section:example_SO_n_2}, one computes that
\begin{equation}
\psi(v)[0]=e^{-\pi Q(v,v)} tr_s(Ne^{\nabla_v^2})[0] = e^{-2\pi h_z(s_v)}
\end{equation}
and hence, writing $v=v_z+v'_z$ with $v_z \in \mathcal{F}_z^2$ and $v'_z \in \overline{\mathcal{F}_z^1}$, we find that
\begin{equation}
\omega(g_\tau)\psi(v)[0](z) = y^{m/4} e^{\pi i Q(v'_z,v'_z)\tau +\pi i Q(v_z,\overline{v_z})\overline{\tau}}
\end{equation}
is the usual Siegel Gaussian $\omega(g_\tau)\varphi_{SG}(v)$. By Proposition~\ref{thm:comparison_varphi_varphi_KM} and \eqref{eq:lowering_op}, we conclude that
\begin{equation}
(2\pi i)^{-1}\partial \overline{\partial} \varphi_{SG} =\omega(X_-) \varphi_{KM}, \quad X_-=\tfrac{1}{2} \left( \begin{smallmatrix} 1 & -i \\ -i & -1 \end{smallmatrix}\right).
\end{equation}
This gives another proof of \cite[Thm. 4.4]{BruinierFunke}.
\end{remark}

\subsection{The currents defined by the forms $\varphi(v_1,\ldots,v_r)$} \label{subsection:Thom_property} We will now study the properties of currents obtained by integration against $\varphi$. For a subspace $U$ of $V_\mathbb{R}$, define $G_U \subset G$ to be the pointwise stabilizer of $U$ and
\begin{equation}
\mathbb{D}_U = \{z \in \mathbb{D}| U \subset \mathcal{F}_z^{w/2}\}.
\end{equation} 
When $U \subset V_\mathbb{Q}$, this is the locus of $z \in \mathbb{D}$ where all classes in $U$ are Hodge classes. It is non-empty only if $U = 0$ or $Q|_{U\times U}$ is positive definite.

\begin{proposition} \label{prop:Thom_property_on_D}
Let $(v_1,\ldots,v_r) \in V_\mathbb{R}^r$ and $U=\langle v_1,\ldots,v_r \rangle$. Assume that $U$ has dimension $r$ and that the restriction of $Q$ to $U$ is positive definite. Let $\eta$ be a compactly supported form on $\mathbb{D}$. Then
\begin{equation}
\lim_{t \to \infty} \int_{\mathbb{D}} \varphi^0(tv_1,\ldots,tv_r) \wedge \eta = (-2\pi i)^{r \cdot rk(\mathcal{F})} \cdot \int_{\mathbb{D}_U} Td(\mathcal{F}^\vee,\nabla)^{-r} \wedge \eta.
\end{equation}
\end{proposition}
\begin{proof}
Note that the hypothesis on $U$ implies that $s_v:=(s_{v_1},\ldots,s_{v_r})$ is a regular section of $(\mathcal{F}^{\oplus r})^\vee$; in particular, the Koszul complex $K(v_1,\ldots,v_r)$ is a resolution of $\mathcal{O}_{\mathbb{D}_U}$. We would like to apply the main result of Bismut \cite{BismutInv90} concerning the limit as $t \to \infty$ of Chern character forms, but this result is for superconnections on a compact complex manifold. We argue using the compact dual $\mathbb{D}^\vee$. Recall that the Hodge bundle $\mathcal{F}$ and section $s_v$ are naturally defined on $\mathbb{D}^\vee$. Using a partition of unity, we construct a hermitian metric $\tilde{h}$ on $\mathcal{F}$ over $\mathbb{D}^\vee$ whose restriction to the support of $\eta$ equals the Hodge metric. Let $\tilde{\nabla}$ be the Chern connection on $(\wedge (\mathcal{F}^{\oplus r}),\tilde{h})$ and consider the superconnection $\tilde{\nabla}_v=\tilde{\nabla}+i \sqrt{2\pi}(s_v+s_v^*)$. By \cite[Thm 3.2]{BismutInv90}, we have
\begin{equation}
\lim_{t \to \infty} tr_s(e^{\tilde{\nabla}_{tv}^2}) = (-2\pi i)^{r \cdot rk(\mathcal{F})} \cdot Td(\mathcal{F}^\vee,\tilde{\nabla})^{-r} \delta_{\mathbb{D}_U}
\end{equation}
as currents on $\mathbb{D}^\vee$, and this implies the result since the restriction of $tr_s(e^{\tilde{\nabla}_{tv}^2})$ to the support of $\eta$ equals $\varphi^0(tv_1,\ldots,tv_r)$.
 \end{proof}

The next proposition is the main ingredient in the proof of Theorem~\ref{thm:Four_expansion}. It shows that $\varphi^0(v_1,\ldots,v_r)$ has a Thom form property in certain quotients of $\mathbb{D}$.

\begin{proposition} \label{prop:Thom_property}
Let $L \subset V_\mathbb{Q}$ be a lattice and $\Gamma$ be a torsion-free subgroup contained in $Stab_{G^0}(L)$. Let $U=\langle v_1,\ldots,v_r \rangle$ be a positive definite subspace of $V$ of dimension $0 \leq r' \leq r$ and write $\Gamma_U = \Gamma \cap G_U^0$ and $\mathbb{D}_U^+=\mathbb{D}_U \cap \mathbb{D}^+$. Let $\eta$ be a closed and compactly supported form on $\Gamma \backslash \mathbb{D}^+$. Then
\begin{equation} \label{eq:Thom_property}
\int_{\Gamma_U \backslash \mathbb{D}^+} \varphi^0(v_1,\ldots,v_r) \wedge \eta = (-2\pi i)^{r' \cdot rk(\mathcal{F})} \int_{\Gamma_U \backslash \mathbb{D}_U^+} ch(\wedge \mathcal{F},\nabla)^{r-r'} \wedge Td(\mathcal{F}^\vee,\nabla)^{-r'} \wedge \eta.
\end{equation}
\end{proposition}
\begin{proof} By Proposition~\ref{prop:phi_main_properties}.(e), the identity \eqref{eq:Thom_property} does not change if we replace $(v_1,\ldots,v_r)$ by $(v_1,\ldots,v_r) \cdot h$ with $h \in SO(r)$. Hence we can assume that $T(v_1,\ldots,v_r)$ is diagonal and, since $\varphi^0(0)=ch(\wedge \mathcal{F},\nabla)$, this reduces the proof to the case $r=r'$. 
Note first that the integral on the left hand side converges: by the estimate \eqref{eq:varphi_estimate}, the sum
\begin{equation}
f^0(v_1,\ldots,v_r):=\sum_{\gamma \in \Gamma_U \backslash \Gamma} \gamma^*\varphi^0(v_1,\ldots,v_r)
\end{equation}
converges and defines a smooth form on $\Gamma \backslash \mathbb{D}^+$, and the integrability follows from Fubini's theorem. Next we claim that the integral
\begin{equation}
I(t,\eta):=\int_{\Gamma_U \backslash \mathbb{D}^+} \varphi^0(tv_1,\ldots,tv_r) \wedge \eta, \quad t >0,
\end{equation}
is independent of $t$. Namely, given $t_2 > t_1>0$, the transgression formula \eqref{eq:transg_formula} shows that $\varphi^0(t_2v_1,\ldots,t_2v_r)-\varphi^0(t_1v_1,\ldots,t_1v_r)=d\psi(v_1,\ldots,v_r)$, with
\begin{equation}
\psi(v_1,\ldots,v_r)=\int_{t_1}^{t_2} tr_s(\sum_{1 \leq i \leq r}(s_{v_i}+s_{v_i}^*)e^{\nabla_{tv_1,\ldots,tv_r}^2}) dt.
\end{equation}
Applying \eqref{eq:varphi_estimate} again, we find that
\begin{equation}
\begin{split}
\int_{\Gamma_U \backslash \mathbb{D}^+} d \psi(v_1,\ldots,v_r) \wedge \eta &= \int_{\Gamma \backslash \mathbb{D}^+} \sum_{\gamma \in \Gamma_U \backslash \Gamma} \gamma^*d \psi(v_1,\ldots,v_r) \wedge \eta \\
&= \int_{\Gamma \backslash \mathbb{D}^+} d \left(\sum_{\gamma \in \Gamma_U \backslash \Gamma} \gamma^* \psi(v_1,\ldots,v_r) \wedge \eta \right) = 0.
\end{split}
\end{equation}
Hence it suffices to show that
\begin{equation} \label{eq:Thom_property_1}
\lim_{t \to \infty} I(t,\eta) = (-2\pi i)^{r \cdot rk(\mathcal{F})} \cdot \int_{\Gamma_U \backslash \mathbb{D}_U^+} Td(\mathcal{F}^\vee,\nabla)^{-r} \wedge \eta.
\end{equation}
Let $(U_\alpha)_{\alpha \in I}$ be a finite open cover of $Supp(\eta)$ with $U_\alpha$ contractible and, for each $\alpha \in I$, pick $V_\alpha \subset \mathbb{D}^+$ such that the quotient map $\pi : \mathbb{D}^+ \to \Gamma \backslash \mathbb{D}^+$ induces an isomorphism $U_\alpha \cong V_\alpha$. Choose a partition of unity $\{\rho_{U_\alpha}| {\alpha \in I}\} \cup \{\rho_\infty\}$ subordinate to the open cover $(U_\alpha)_{\alpha \in I} \cup (\Gamma \backslash \mathbb{D}^+- Supp(\eta))$ and write $\eta = \Sigma_{\alpha \in I} \rho_{U_\alpha}\eta$. Setting $\rho_{V_\alpha} = \pi^*(\rho_{U_\alpha}) \cdot 1_{V_\alpha} \in A_c^0(\mathbb{D}^+)$, we find that
\begin{equation}
\begin{split}
I(t,\eta) &= \int_{\Gamma \backslash \mathbb{D}^+} f^0(tv_1,\ldots,tv_r) \wedge \eta \\
&= \sum_{\alpha \in I} \int_{U_\alpha} f^0(tv_1,\ldots,tv_r) \wedge \rho_{U_\alpha} \eta \\
&= \sum_{\alpha \in I} \int_{\mathbb{D}^+} f^0(tv_1,\ldots,tv_r) \wedge \rho_{V_\alpha} \pi^*\eta.
\end{split}
\end{equation}
For fixed $\alpha \in I$, write $\Gamma_U \backslash \Gamma=S_1 \sqcup S_2$ (disjoint union), with
\begin{equation}
S_1=\{\gamma \in \Gamma_U \backslash \Gamma | h_z(\gamma^{-1} v_1,\ldots,\gamma^{-1} v_r) \leq 1, \text{ for some } z \in V_\alpha \};
\end{equation}
then $S_1$ is a finite set since the form $q_z(v)$ in \eqref{eq:def_q_z} is positive definite. The estimate \eqref{eq:varphi_estimate} shows that
\begin{equation}
\lim_{t \to \infty} \int_{\mathbb{D}^+} \sum_{\gamma \in S_2} \gamma^* \varphi^0(tv_1,\ldots,tv_r) \wedge \rho_{V_\alpha} \pi^*\eta = 0
\end{equation}
by dominated convergence. For the sum over $S_1$, Proposition~\ref{prop:Thom_property_on_D} gives
\begin{equation}
\begin{split}
\lim_{t \to \infty} \sum_{\gamma \in S_1} &\int_{\mathbb{D}^+}  \varphi^0 (t\gamma^{-1}v_1,\ldots,t\gamma^{-1}v_r) \wedge \rho_{V_\alpha} \pi^*\eta \\
&=  (-2\pi i)^{r \cdot rk(\mathcal{F})}\sum_{\gamma \in S_1} \int_{\mathbb{D}^+_{\gamma^{-1}U}} Td(\mathcal{F}^\vee,\nabla)^{-r} \wedge \rho_{V_\alpha} \pi^*\eta \\
&=  (-2\pi i)^{r \cdot rk(\mathcal{F})}\sum_{\gamma \in \Gamma_U \backslash \Gamma} \int_{\mathbb{D}^+_{\gamma^{-1}U}} Td(\mathcal{F}^\vee,\nabla)^{-r} \wedge \rho_{V_\alpha} \pi^*\eta \\
&= (-2\pi i)^{r \cdot rk(\mathcal{F})} \sum_{\gamma \in \Gamma_U \backslash \Gamma} \int_{\mathbb{D}^+_U} Td(\mathcal{F}^\vee,\nabla)^{-r} \wedge (\gamma^{-1})^* \rho_{V_\alpha} \pi^*\eta \\
&=(-2\pi i)^{r \cdot rk(\mathcal{F})} \sum_{\gamma \in \Gamma_U \backslash \Gamma} \int_{\Gamma_U \backslash \mathbb{D}^+_U} Td(\mathcal{F}^\vee,\nabla)^{-r} \wedge \left( \sum_{\gamma' \in \Gamma_U} (\gamma')^*(\gamma^{-1})^* \rho_{V_\alpha} \right) \eta \\
&=(-2\pi i)^{r \cdot rk(\mathcal{F})} \int_{\Gamma_U \backslash \mathbb{D}^+_U} Td(\mathcal{F}^\vee,\nabla)^{-r} \wedge \left( \sum_{\gamma \in \Gamma} \gamma^* \rho_{V_\alpha} \right) \eta \\
&=(-2\pi i)^{r \cdot rk(\mathcal{F})} \int_{\Gamma_U \backslash \mathbb{D}^+_U} Td(\mathcal{F}^\vee,\nabla)^{-r} \wedge \rho_{U_\alpha} \eta
\end{split}
\end{equation}
and \eqref{eq:Thom_property_1} follows by summing over $\alpha \in I$.
 \end{proof}

\section{Hodge loci in arithmetic quotients of period domains and Siegel modular forms}

We now consider quotients $\Gamma \backslash \mathbb{D}^+$ of period domains by arithmetic groups $\Gamma \subset G^0$. After defining special cycles in \textsection \ref{subsection:special_cycles}, we introduce theta series valued in the space of differential forms $A^*(\Gamma \backslash \mathbb{D}^+)$ by summing the forms $\varphi(v)$ for $v$ varying in an appropriate lattice (\textsection \ref{subsection:theta_series_def}). Using the results of \autoref{section:Properties_of_varphi}, we show that the resulting differential forms are closed, and that their cohomology classes define Siegel modular forms (\autoref{thm:modularity_period_domains}). In \textsection \ref{subsection:theta_series_Four_exp}, we compute the Fourier expansion of these modular forms in terms of the special cycles mentioned above (\autoref{thm:Four_expansion}).

\subsection{Special cycles in arithmetic quotients of period domains} \label{subsection:special_cycles} Let $L$ be an even integral lattice in $V_\mathbb{Q}$ and denote by $L^\vee \supset L$ its dual lattice. The arithmetic group
\begin{equation}
\Gamma_L := \{\gamma \in G^0 \ | \ \gamma(L)=L, \ \gamma|_{L^\vee/L} = id \}
\end{equation}
acts properly discontinuously on $\mathbb{D}$ by holomorphic transformations, and we write $X_L:=\Gamma_L \backslash \mathbb{D}^+$ for the quotient. We assume that $\Gamma_L$ is neat. Given a subspace $\langle v_1,\ldots,v_r \rangle$ of $V_\mathbb{Q}$, we write $\Gamma_{L,{\langle v_1,\ldots, v_r \rangle}}$ for the pointwise stabilizer of $\langle v_1,\ldots,v_r \rangle$ in $\Gamma_L$. The group $\Gamma_{L,\langle v_1,\ldots,v_r \rangle}$ stabilizes $\mathbb{D}_{\langle v_1,\ldots,v_r \rangle}^+=\mathbb{D}_{\langle v_1,\ldots,v_r \rangle} \cap \mathbb{D}^+$. This yields a commutative diagram
\begin{equation}
\xymatrix{ \mathbb{D}^+_{\langle v_1,\ldots, v_r \rangle} \ar[r] \ar[d] & \mathbb{D}^+ \ar[d] \\ \Gamma_{L,{\langle v_1,\ldots, v_r \rangle}} \backslash \mathbb{D}^+_{\langle v_1,\ldots, v_r \rangle} \ar[r]  & X_L }
\end{equation}
where all the maps are holomorphic and the horizontal ones are proper.
Define 
\begin{equation}
Hdg(v_1,\ldots,v_r)_L := [\Gamma_{L,{\langle v_1,\ldots, v_r \rangle}} \backslash \mathbb{D}^+_{\langle v_1,\ldots, v_r \rangle}] \in \mathrm{H}^{2*}(X_L)
\end{equation}
to be the cycle class of $\Gamma_{L,{\langle v_1,\ldots, v_r \rangle}} \backslash \mathbb{D}^+_{\langle v_1,\ldots, v_r \rangle}$ in the singular cohomology of $X_L$.

The cycle $Hdg(v_1,\ldots,v_r)_L$ only depends on the orbit of $\langle v_1,\ldots, v_r \rangle$ under $\Gamma_L$. Following Kudla \cite{KudlaOrthogonal}, we now define certain sums of the classes $Hdg(v_1,\ldots,v_r)_L$ as follows. Fix a symmetric matrix $T \in Sym_{r}(\mathbb{Q})$ and a class $\mu + L^r \in (L^\vee /L)^r$ and consider the set 
\begin{equation}
L(T,\mu)=\{(v_1,\ldots,v_r) \in \mu+L^r \ | \ Q(v_i,v_j)=2T \}.
\end{equation}
The group $\Gamma_L$ acts on $L(T,\mu)$ with finitely many orbits. We define
\begin{equation}
Hdg(T,\mu) = \sum_{(v_1,\ldots,v_r) \in \Gamma_L \backslash L(T,\mu)} Hdg(v_1,\ldots,v_r)_L \in H^{2*}(X_L).
\end{equation}
The points in the support of $Hdg(T,\mu)$ correspond to $\Gamma_L$-orbits of Hodge structures on $V_\mathbb{Q}$ containing Hodge classes $v_1,\ldots,v_r$ such that $(v_1,\ldots,v_r) \in \mu + L^r$ and $Q(v_i,v_j)=2T$. In particular, we have $Hdg(T,\mu)=0$ unless the quadratic form $T$ is represented by $(V_\mathbb{Q},Q)$ and (by the second Riemann bilinear relation) the matrix $T$ is positive semidefinite.

\subsection{Theta series and differential forms on $X_L$} \label{subsection:theta_series_def} Let $g= (n(x)m(a),1) \in P$. By Proposition~\ref{prop:phi_main_properties}.(e), the form $\omega(g)\varphi(v_1,\ldots,v_r)$ depends only on the image $\tau = g \cdot i$ in the Siegel space $\mathfrak{H}_r$ of genus $r$. We denote this value by $\omega(g_\tau)\varphi(v_1,\ldots,v_r)$.

\begin{definition}
Let $r$ be a positive integer and $\mu \in (L^\vee/L)^r$. For $\tau=x+iy \in \mathfrak{H}_r$, define
\begin{equation}
\theta(\tau ; \mu+L^r) = \det(y)^{-m/4} \sum_{(v_1,\ldots,v_r) \in \mu + L^r} \omega(g_\tau)\varphi(v_1,\ldots,v_r).
\end{equation}
\end{definition}
Note that the sum and all its partial derivatives converge normally by Proposition~\ref{prop:superconn_Schwartz} and the estimates before it. Hence it defines a smooth form
\begin{equation}
\theta(\tau ; \mu+L^r) \in \oplus_{p \geq 0} A^{p,p}(X_L)
\end{equation}
that is closed by Proposition~\ref{prop:phi_main_properties}. We denote by
\begin{equation}
[\theta(\tau ; \mu+L^r)] \in H^{2*}(X_L)
\end{equation}
its cohomology class. 

Let $Mp_{2r,\mathbb{Z}}$ be the inverse image of $Sp_{2r}(\mathbb{Z})$ under the double cover $Mp_{2r,\mathbb{R}} \to Sp_{2r}(\mathbb{R})$. 
We will prove shortly that $[\theta(\cdot ; \mu+L^r)]$ is a holomorphic Siegel modular form of weight $m/2$ and some level $\Gamma \subset Mp_{2r,\mathbb{Z}}$. A more precise result may be proved by combining all the series $\theta(\cdot ; \mu+L^r)$ for varying $\mu$ in a theta series 
\begin{equation}
\theta_{L^r}: \mathfrak{H}_r \to A^*(X_L) \otimes \mathbb{C}[(L^\vee/L)^r]^\vee
\end{equation}
valued in the dual of the group algebra $\mathbb{C}[(L^\vee/L)^r]$, defined by $\theta_{L^r}(\mu) = \theta(\cdot;\mu+L^r)$.

Let us recall the action of the group $Mp_{2r,\mathbb{Z}}$ on the group algebra $\mathbb{C}[(L^\vee/L)^r]$ via the finite Weil representation $\rho_L$. For this and for the proof of \autoref{thm:modularity_period_domains}, it is convenient to work with adeles. 

We write $Mp_{2r,\mathbb{A}}$ for the metaplectic double cover of $Sp_{2r}(\mathbb{A})$ and $K'$ for the inverse image in $Mp_{2r,\mathbb{A}}$ of $Sp_{2r}(\hat{\mathbb{Z}})$. There is a canonical splitting $Sp_{2r}(\mathbb{Q}) \to Mp_{2r,\mathbb{A}}$ of the double cover $Mp_{2r,\mathbb{A}} \to Sp_{2r}(\mathbb{A})$; identifying $Sp_{2r}(\mathbb{Q})$ with its image in $Mp_{2r,\mathbb{A}}$, we have
\begin{equation}
Sp_{2r}(\mathbb{Q}) \cap (K'Mp_{2r,\mathbb{R}}) = Sp_{2r}(\mathbb{Z}).
\end{equation}
We obtain a homomorphism $Mp_{2r,\mathbb{Z}} \to K'$ sending $\gamma$ to $\hat{\gamma}$, where $\hat{\gamma}$ is the unique element of $K'$ such that $\gamma \hat{\gamma}$ is in the image of $Sp_{2r}(\mathbb{Q})$.

Denote by $\omega=\omega_\psi$ the Weil representation of $Mp_{2r,\mathbb{A}}$ on $\mathcal{S}(V(\mathbb{A})^r)$ attached to the standard additive character $\psi:\mathbb{A} \to \mathbb{C}^\times$ with $\psi_\infty(x)=e^{2\pi ix}$. We realize the group algebra $\mathbb{C}[(L^\vee/L)^r]$ as the subspace $\mathcal{S}_{L^r}$ of Schwartz functions in $\mathcal{S}(V(\mathbb{A}_f)^r)$ that are supported on $(L^\vee)^r \otimes \hat{\mathbb{Z}}$ and are constant on cosets of $L^r \otimes \hat{\mathbb{Z}}$. Then $\mathcal{S}_{L^r}$ is stable under the action of $K'$, and restriction to $Mp_{2r,\mathbb{Z}}$ via the above splitting defines the representation $\rho_L$.

We can now consider Siegel modular forms valued in the dual representation $\rho^{\vee}_L$. Given $\gamma=\left[\left( \begin{smallmatrix} a & b \\ c & d \end{smallmatrix} \right),\pm 1 \right] \in Mp_{2r,\mathbb{Z}}$ and $\tau \in \mathfrak{H}_r$, choose $g_\tau, g_{\gamma \tau} \in P$ such that $g_\tau \cdot i = \tau$ and $g_{\gamma \tau} \cdot i=\gamma \tau$; then we have
\begin{equation}
\gamma g_\tau = g_{\gamma \tau} k(\gamma,\tau),
\end{equation}
where $k(\gamma,\tau) \in \widetilde{U(r)}$ and the value of $\det(k(\gamma,\tau))^{1/2}$ is independent of the choices of $g_\tau$, $g_{\gamma \tau}$. For $m$ a positive integer, define an automorphy factor
\begin{equation}
j_{m/2}:Mp_{2r,\mathbb{Z}} \times \mathfrak{H}_r \to \mathbb{C}^\times, \quad j_{m/2}(\gamma,\tau) = \det(k(\gamma,\tau))^{m/2} \cdot |c\tau + d|^{m/2}.
\end{equation}
We say that a holomorphic function $f: \mathfrak{H}_r \to \mathbb{C}[(L^{\vee}/L)^r]^\vee$ is a Siegel modular form of weight $m/2$ and level $1$ if
\begin{equation}
f(\gamma \tau) = j_{m/2}(\gamma,\tau) \cdot \rho^{\vee}_L(\gamma) f(\tau), \quad \gamma=\left[\left( \begin{smallmatrix} a & b \\ c & d \end{smallmatrix} \right),\pm 1 \right] \in Mp_{2r,\mathbb{Z}}, \quad \tau \in \mathfrak{H}_r.
\end{equation}
(When $r=1$ we also require that $f$ be holomorphic at the cusp $i \infty$.) We denote the space of such modular forms by $M^{(r)}_{m/2,\rho_L}$.

\begin{theorem} \label{thm:modularity_period_domains} The cohomology class $[\theta_{L^r}]$ belongs to $M^{(r)}_{m/2,\rho_L} \otimes H^{2*}(X_L)$.
\end{theorem}

\begin{proof}
For $\varphi_f \in \mathcal{S}(V(\mathbb{A}_f)^r)$ and $g_\mathbb{A}=(g_f,g_\infty) \in Mp_{2r,\mathbb{A}}$, consider the theta function
\begin{equation}
\theta(g_\mathbb{A};\varphi_f) = \sum_{v \in V(\mathbb{Q})^r} \omega(g_f,\cdot)\varphi_f(v) \cdot \omega(g_\infty)\varphi(v) \in \oplus_{p \geq 0} A^{p,p}(G(\mathbb{Q}) \backslash (G(\mathbb{A}_f) \times \mathbb{D}))
\end{equation}
and denote by $[\theta(g_\mathbb{A};\varphi_f)]$ its cohomology class. By Poisson summation we have
\begin{equation}
\theta(\gamma g_\mathbb{A};\varphi_f) = \theta(g_\mathbb{A};\varphi_f), \ \gamma \in Sp_{2r}(\mathbb{Q})
\end{equation}
(see e.g. Tate's thesis, Lemma 4.2.4.). Moreover, by Lemma~\ref{lemma:weight} we have
\begin{equation}
r(X)\theta(g_\mathbb{A};\varphi_f) = \frac{m}{2}\chi(X)\theta(g_\mathbb{A};\varphi_f) + d\theta_{\nu(X)}(g_\mathbb{A};\varphi_f)
\end{equation}
for any $X \in \mathfrak{u}(r)$, where
\begin{equation}
\theta_{\nu(X)}(g_\mathbb{A};\varphi_f):= \sum_{v \in V(\mathbb{Q})^r} \omega(g_f,\cdot)\varphi_f(v) \cdot \omega(g_\infty)\nu(X,v) \in A^*(G(\mathbb{Q}) \backslash (G(\mathbb{A}_f) \times \mathbb{D})).
\end{equation}
Taking cohomology classes, we conclude that
\begin{equation}
[\theta(g_\mathbb{A} k;\varphi_f)] = \det(k)^{m/2} \cdot [\theta(g_\mathbb{A};\varphi_f)], \ k \in \widetilde{U(r)}.
\end{equation}
Next we check that $[\theta(\cdot;\varphi_f)]$ is holomorphic. Let
\begin{equation}
\mathfrak{p}_{\pm} = \left\{ \left. p_{\pm}(X) = \frac{1}{2} \left( \begin{array}{cc} X & \pm iX \\ \pm iX & -X \end{array} \right) \right| X = {^t X} \in M_r(\mathbb{R}) \right\}.
\end{equation}
Then we have the Harish-Chandra decomposition
\begin{equation}
\mathfrak{sp}(2r)_\mathbb{C} = \mathfrak{u}(r) \oplus \mathfrak{p}_+ \oplus \mathfrak{p}_-
\end{equation}
and an identification
\begin{equation}
T^{0,1}\mathfrak{H}_r \cong Sp_{2r}(\mathbb{R}) \times_{U(r)} \mathfrak{p}_-
\end{equation}
of $Sp_{2r}(\mathbb{R})$-equivariant bundles on $\mathfrak{H}_r$. We need to show that the Cauchy-Riemann equations
\begin{equation}
X_-[\theta((1,g_\tau);\varphi_f)] = 0 \text{ for every } X_-=p_-(X) \in \mathfrak{p}_-
\end{equation}
hold. Assume first that $r=1$. Writing $\theta(\tau,\varphi_f) = y^{-m/4} \theta((1,g_\tau),\varphi_f)$, the operator $X_-:=p_-(1)$ is given by
\begin{equation}
y^{-\frac{1}{2}(\frac{m}{2}-2)} X_- [\theta((1,g_{\tau}),\varphi_f)] = -2i y^2 \tfrac{d}{d\overline{\tau}} [\theta(\tau,\varphi_f)].
\end{equation}
Setting
\begin{equation}
\theta_\psi(\tau;\varphi_f)=y^{-(m/4-1)} \cdot \sum_{v \in V(\mathbb{Q})} \omega(\cdot)\varphi_f(v) \cdot \omega(g_\tau)\psi(v) \in \oplus_{p \geq 0} A^{p,p}(G(\mathbb{Q}) \backslash (G(\mathbb{A}_f) \times \mathbb{D})),
\end{equation}
we find that $2iy^2\tfrac{d}{d\overline{\tau}} \theta(\tau,\varphi_f)=\partial \overline{\partial} \theta_\psi(\tau;\varphi_f)$ by Lemma~\ref{lemma:lowering_varphi}. Thus $\tfrac{d}{d\overline{\tau}} \theta(\tau,\varphi_f)$ is exact and $X_-[\theta(g_\tau;\varphi_f)] = 0$.

For general $r$, since $X$ is symmetric, we can find $k \in SO(r)$ such that $Ad(k) p_-(X) = p_-(Y)$ with $Y$ diagonal. By Proposition~\ref{prop:phi_main_properties}.(e) and \eqref{eq:varphi_wedge}, this reduces to the case $r=1$.

Finally, the holomorphicity at $i\infty$ when $r=1$ follows from the estimates \eqref{eq:varphi_estimate} that show that each $\theta(\tau,\mu+L)$ is of moderate growth with respect to $\tau$.
 \end{proof}

\begin{remarks} \begin{enumerate} \item Let $\Gamma'$ be the inverse image of the congruence subgroup $\Gamma(4|L^\vee/L|) $ of $Sp_{2r}(\mathbb{Z})$ in $Mp_{2r,\mathbb{Z}}$. The restriction of $\rho_L$ to $\Gamma'$ is trivial, and in particular we obtain that
\begin{equation}
[\theta(\cdot, \mu+L^r)] \in M^{(r)}_{m/2}(\Gamma') \otimes H^{2*}(X_L).
\end{equation}
\item For general Hodge weights, it is known that the complex manifold $X_L$ is not algebraic; see \cite[Thm 1.1]{GriffithsRoblesToledo}. One can define the Bott-Chern cohomology groups of a complex manifold $X$ as follows:
\begin{equation}
H_{BC}^{p,q}(X)=\frac{\{ \omega \in A^{p,q}(X) | d\omega = 0\} }{\partial \overline{\partial} A^{p-1,q-1}(X)}.
\end{equation}
For a general (non-K\"ahler) complex manifold, these groups give a strictly finer invariant than the de Rham cohomology groups $H^{2*}(X)$; the author does not know if this is the case for (some of) the complex manifolds $X_L$. In any case, as explained e.g. in \cite[\textsection 4.1]{BismutHLBCC}, Chern classes of holomorphic vector bundles lie naturally in $H^*_{BC}(X)$. Replacing the use of Lemma~\ref{lemma:weight} by \eqref{eq:strong_weight_behaviour} in the proof of the theorem, we find that the Bott-Chern cohomology class $[\theta_{L^r}]_{BC}$ of $\theta_{L^r}$ satisfies
\begin{equation}
[\theta_{L^r}]_{BC} \in M^{(r)}_{m/2,\rho_L} \otimes \oplus_{p \geq 0} H^{p,p}_{BC}(X_L).
\end{equation}

\item One can try to use the theta functions $[\theta_{L^r}]$ to define theta lifts of cusp forms on $Mp_{2r}$ to classes in the coherent cohomology of $X_L$. It would be interesting to understand when such lifts are non-vanishing and study their properties, particularly with respect to cup products.
\end{enumerate}
\end{remarks}

\subsection{The Fourier expansion of $[\theta_{L^r}(\tau)]$} \label{subsection:theta_series_Four_exp} Recall that holomorphic Siegel modular forms $f : \mathfrak{H}_r \to \mathbb{C}$ of level $\Gamma \subset Sp_{2r}(\mathbb{Z})$ have a Fourier expansion
\begin{equation}
f(\tau) = \sum_{T \in Sym_{r}(\mathbb{Q})_{\geq 0}} a_T \cdot q^T, \quad q^T:=e^{2\pi i tr(T\tau)},
\end{equation}
where the sum runs over positive semidefinite matrices $T$ with rational entries. We have
\begin{equation}
a_T \cdot e^{-2\pi tr(Ty)} = \int_{(Sym_r(\mathbb{R}) \cap \Gamma) \backslash Sym_r(\mathbb{R})} f(x+iy) e^{-2\pi i tr(Tx)} dx, \ y \in Sym_r(\mathbb{R})_{>0},
\end{equation}
where $dx$ is the Haar measure on $Sym_r(\mathbb{R})$ giving $(Sym_r(\mathbb{R}) \cap \Gamma) \backslash Sym_r(\mathbb{R})$ unit volume. Here is the result computing the Fourier expansion of $[\theta_{L^r}(\tau)]$; it implies \autoref{thm:thm_intro_1_2} and \eqref{thm:thm_1_intro}.

\begin{theorem} \label{thm:Four_expansion}
\begin{equation} \label{eq:main_theorem}
[\theta_{L^r}(\tau)^*] = Td(\mathcal{F}^\vee)^{-r} \cup \sum_{\mu \in (L^\vee/L)^r} \sum_{T \in Sym_r(\mathbb{Q})_{\geq 0}} Hdg(T,\mu) \cup c^{top}(\mathcal{F}^\vee)^{r-rk(T)} \cdot q^T.
\end{equation} 
\end{theorem}
\begin{proof}
For $T \in Sym_{r}(\mathbb{Q})_{\geq 0}$ and $y=a {^t a}$, we have
\begin{equation}
\begin{split}
\theta(y,\mu+L^r)_T &:= e^{2\pi tr(Ty)} \int_{(Sym_r(\mathbb{R}) \cap \Gamma) \backslash Sym_r(\mathbb{R})} \theta(x+iy,\mu+L^r) e^{-2\pi i tr(Tx)} dx \\
&= e^{2\pi tr(Ty)} \det(y)^{-m/4} \cdot \sum_{(v_1,\ldots,v_r) \in L(T,\mu)} \omega(m(a))\varphi(v_1,\ldots,v_r) \\
&= \sum_{(v_1,\ldots,v_r) \in L(T,\mu)} \varphi^0((v_1,\ldots,v_r)\cdot a) \\
&= \sum_{(v_1,\ldots,v_r) \in \Gamma_L \backslash L(T,\mu)} \sum_{\gamma \in \Gamma_{L,\langle v_1,\ldots,v_r \rangle} \backslash \Gamma_L} \varphi^0((\gamma^{-1}v_1,\ldots,\gamma^{-1}v_r)\cdot a),
\end{split}
\end{equation}
where the sum over $(v_1,\ldots,v_r) \in \Gamma_L \backslash L(T,\mu)$ is finite. 
Let $\eta$ be a closed and compactly supported form on $X_L$. We find that
\begin{equation}
\int_{X_L} \theta(y,\mu+L^r)^*_T \wedge \eta = \sum_{(v_1,\ldots,v_r) \in \Gamma_L \backslash L(T,\mu)} I((v_1,\ldots,v_r) \cdot a),
\end{equation}
with
\begin{equation}
I(v_1,\ldots,v_r) = \int_{\Gamma_{L,\langle v_1,\ldots,v_r \rangle} \backslash \mathbb{D}^+} \varphi^0(v_1,\ldots,v_r)^* \wedge \eta.
\end{equation}
Note that for $(v_1,\ldots,v_r) \in \Gamma_L \backslash L(T,\mu)$, we have $\dim \langle v_1,\ldots,v_r \rangle=rk(T)$. By Proposition~\ref{prop:Thom_property}, we have
\begin{equation}
I(v_1,\ldots,v_r) = \int_{\Gamma_{L,\langle v_1,\ldots,v_r \rangle} \backslash \mathbb{D}^+_{\langle v_1,\ldots,v_r \rangle}} (c^{top}(\mathcal{F}^\vee,\nabla)^*)^{r-rk(T)} \wedge (Td(\mathcal{F}^\vee,\nabla)^*)^{-r} \wedge \eta. 
\end{equation}
We obtain that
\begin{equation}
[\theta(y,\mu+L^r)^*_T] = Hdg(T,\mu) \cup c^{top}(\mathcal{F}^\vee)^{r-rk(T)} \cup Td(\mathcal{F}^\vee)^{-r} \in H^{2*}(X_L),
\end{equation}
and the theorem follows.
 \end{proof}

\begin{remark} 
When $\mathbb{D}$ is a hermitian symmetric domain of type IV, the cycles $Hdg(T,\mu)$ agree with the special cycles considered in \cite{KudlaMillson3}, and the above theorem implies Theorem 2 in that paper.
\end{remark}

\appendix\normalsize

\section{Superconnections} \label{appendix}

In this appendix we have gathered relevant facts on Quillen's formalism of superconnections and related subjects. For more details and proofs, see \cite{BGV,QuillenChern}.

\subsection{Superconnections and Chern forms}
A super vector space is a complex $\mathbb{Z}/2\mathbb{Z}$-graded vector space; a super algebra is a complex $\mathbb{Z}/2\mathbb{Z}$-graded algebra. The graded tensor product of two super vector spaces (denoted by $\hat{\otimes}$) is again a super vector space; in particular, if $V=V_0 \oplus V_1$ is a super vector space, then $End(V)$ is also a super vector space. We write $\tau$ for the endomorphism of $V$ determined by $\tau(v)=(-1)^{deg(v)}v$. The supertrace $tr_s: End(V) \to \mathbb{C}$ is the linear form defined by
\begin{equation}
tr_s(u)=tr(\tau u),
\end{equation}
where $tr$ denotes the usual trace. If $u=\left(\begin{smallmatrix} a & b \\ c & d \end{smallmatrix}\right)$ with $a \in End(V_0)$, $d \in End(V_1)$, $b \in Hom(V_1,V_0)$ and $c \in Hom(V_0,V_1)$, then $tr_s(u)=tr(a)-tr(d)$; in particular, $tr_s$ vanishes on (super)commutators and on odd endomorphisms. 

Given a super vector bundle $E=E_0 \oplus E_1$ over a differentiable manifold $X$, define a superalgebra $\mathcal{A}$ by
\begin{equation}
\mathcal{A}=A(X,End(E))=A^*(X) \hat{\otimes}_{\mathcal{C}^\infty(X)} \Gamma(End(E)),
\end{equation}
where the tensor product is taken in the sense of superalgebras, that is, we use the Koszul rule of signs to define the product: 
\begin{equation} \label{eq:Koszul_rule_of_signs}
(\omega \otimes u) \cdot (\eta \otimes v) = (-1)^{deg(\eta)deg(u)}(\omega \wedge \eta) \otimes uv, \quad \omega, \eta \in A^*(X), \quad u, v \in \Gamma(End(E)).
\end{equation} 
The supertrace $tr_s: \Gamma(End(E)) \to \mathcal{C}^\infty(X)$ admits a unique extension 
\begin{equation}
tr_s:\mathcal{A} \to A^*(X)
\end{equation}
satisfying $tr_s(\omega \otimes u) = \omega tr_s(u)$. Clearly $\mathcal{A}$ is a left $A^*(X)$-module and $tr_s$ is a map of left $A^*(X)$-modules. Since $tr_s:End(E) \to \mathbb{C}$ vanishes on odd endomorphisms, we conclude that in fact $tr_s:\mathcal{A} \to A^{2*}(X)$.

For a super vector bundle $E$ over $X$, write $A(E) = A^*(X) \hat{\otimes}_{\mathcal{C}^\infty(X)} \Gamma(E)$. A superconnection $\nabla$ on $E$ is an odd element of $End(A(E))$ satisfying the Leibniz rule
\begin{equation}
\nabla(\omega \otimes u) = d\omega \otimes u + (-1)^{deg(\omega)} \omega \wedge \nabla u.
\end{equation}
Note that $A(E)$ is naturally a left $A^*(X)$-module. A direct computation shows that $\nabla^2 \in End(A(E))$ is $A^*(X)$-linear. Quillen \cite{QuillenChern} shows that the natural action of $\mathcal{A}$ on $A(E)$ identifies $\mathcal{A}$ with the $A^*(X)$-linear endomorphisms of $A(E)$; thus we can regard $\nabla^2$ as an element of $\mathcal{A}$ and define the Chern character form
\begin{equation}
ch(E,\nabla) = tr_s(e^{\nabla^2}) \in A^{2*}(X).
\end{equation}
It is clear that this is class is functorial: if $f: X' \to X$ is a map of differentiable manifolds, then 
\begin{equation}
ch(f^*E,f^*\nabla) = f^*(ch(E,\nabla)).
\end{equation}

\label{subsection:App_Chern_form_properties} Denote by $ch(E_0), ch(E_1) \in H^{2*}(X)$ the Chern character of the vector bundles $E_0$ and $E_1$. Let $^*$ be the operator acting by $(-2\pi i)^{-p}$ on $A^{p,p}(X)$. Quillen shows that the form $ch(E,\nabla)$ has the following properties:
\begin{itemize}
\item $ch(E,\nabla)$ is closed.
\item \label{eq:Chern_cohom} $[ch(E,\nabla)^*] = ch(E_0)-ch(E_1)$ (in particular, $[ch(E,\nabla)]$ is independent of $\nabla$).
\item \label{eq:Chern_additive} $ch(E \oplus E', \nabla \oplus \nabla')=ch(E,\nabla) + ch(E',\nabla')$.
\item \label{eq:Chern_product} $ch(E \otimes E', \nabla \otimes 1 + 1 \otimes \nabla') = ch(E, \nabla) \wedge ch(E',\nabla')$.
\end{itemize}

If $\nabla_i$ is a connection on $E_i$ ($i=0,1$) and $u$ is an odd element of $End(E)$, then $\nabla_0+\nabla_1+u$ is a superconnection on $E$; we will only deal with superconnections of this form. Note that if $u=0$, then
\begin{equation} \label{eq:Chern_u_0}
ch(E,\nabla)^* = ch(E_0,\nabla_0)-ch(E_1,\nabla_1),
\end{equation}
where $ch(E_j,\nabla_j)=tr(e^{\frac{i}{2\pi} \nabla_j^2})$, $j=1,2$, are the Chern-Weil forms.

For a smooth family of superconnections $\nabla_t$ on $E$ depending on a real parameter $t$, we have the transgression formula
\begin{equation} \label{eq:transg_formula}
\frac{d}{dt} tr_s(e^{\nabla_t^2}) = d \ tr_s(\frac{d\nabla_t}{dt} e^{\nabla_t^2})
\end{equation}
(\cite[Prop. 1.41]{BGV}).

\subsection{Hermitian holomorphic complexes over complex manifolds}
Some additional properties hold for $ch(E,\nabla)$ when $X$ is a complex manifold and $\nabla$ arises from a complex of holomorphic vector bundles. Namely, consider such a complex
\begin{equation}
0 \to E_0 \overset{v}{\to} \cdots \overset{v}{\to} E_m \to 0
\end{equation}
with $v$ holomorphic. Suppose that each $E_i$ carries a hermitian metric $h_i$, and denote by $v^*:E_{i+1} \to E_i$ the maps adjoint to $v:E_i \to E_{i+1}$. Then $E:=\oplus_{0 \leq i \leq m}E_i$ is a super vector bundle, with even part $\oplus_{i \text{ even}}E_i$ and odd part $\oplus_{i \text{ odd}}E_i$. We endow $E$ with the metric given by the orthogonal sum of the $h_i$ and denote by $\nabla = \nabla^{\text{even}} + \nabla^{\text{odd}}$ the associated Chern connection. Then
\begin{equation}
\nabla_v:=\nabla+i(v+v^*)
\end{equation}
is a superconnection on $E$. Writing $\nabla^{1,0}$ (resp. $\nabla^{0,1}$) for the decomposition of $\nabla$ into holomorphic (resp. antiholomorphic) parts, we find that $(\nabla^{0,1}+iv)^2=0$ since $v$ is holomorphic, and similarly that $(\nabla^{1,0}+iv^*)^2=0$ since $\nabla$ is unitary. This shows that 
\begin{equation}
ch(E,\nabla_v) = tr_s(e^{\nabla_v^2}) \in \oplus_{p \geq 0} A^{p,p}(X) \subset A^{2*}(X).
\end{equation}
Given another complex $0 \to E'_0 \overset{v'}{\to} \cdots \overset{v'}{\to} E'_m \to 0$ of hermitian holomorphic vector bundles with corresponding superconnection $\nabla_{v'}$, consider the total tensor product
\begin{equation}
E''_0 \overset{v''}{\to} \cdots \overset{v''}{\to} E''_{2m} = (E_0 \overset{v}{\to} \cdots \overset{v}{\to} E_m) \otimes (E'_0 \overset{v'}{\to} \cdots \overset{v'}{\to} E'_m).
\end{equation}
Then the super vector bundle $E''=\oplus_{0 \leq i \leq 2m}E''_i$ is the tensor product $E \hat{\otimes} E'$. We have $\nabla_{v''} = \nabla_v \otimes 1 + 1 \otimes \nabla_{v'}$ and hence an identity of Chern forms
\begin{equation} \label{eq:app_tensor_Chern}
ch(E \hat{\otimes} E', \nabla_{v''}) = ch(E,\nabla_{v}) \wedge ch(E',\nabla_{v'}).
\end{equation}

\subsection{Koszul complexes} \label{subsection:Koszul_metric}
Let $E$ be a holomorphic vector bundle on a complex manifold $X$ and let $s: E \to \mathcal{O}_X$ be a holomorphic map. The Koszul complex $K(s)$ associated with $s$ is the complex with underlying vector bundle $\wedge E$ and differential $\wedge^k E \to \wedge^{k-1} E$ defined by
\begin{equation}
d(e_1\wedge \cdots \wedge e_k) = \sum_{1 \leq i \leq k} (-1)^{i+1} s(e_i) \cdot e_1 \wedge \cdots \hat{e_i} \cdots \wedge e_k.
\end{equation}
Given maps $s_i : E \to \mathcal{O}_X$ ($1 \leq i \leq r$), we denote by $K(s_1,\ldots,s_r)$ the Koszul complex associated with the map $(s_1,\ldots,s_r):E^r \to \mathcal{O}_X$; there is an isomorphism
\begin{equation} \label{eq:app_tensor_Koszul} 
K(s_1,\ldots,s_r) \cong \otimes_{1 \leq i \leq r} K(s_i),
\end{equation}
where the right hand side denotes the total tensor product. Given another complex $s':E' \to \mathcal{O}_X$, a map $\alpha: E \to E'$ such that $s=s' \circ \alpha$ induces a morphism $K(\alpha):K(s) \to K(s')$. This morphism is functorial in $\alpha$; in particular, a matrix $(h_{ij}) \in GL_r(\mathbb{C})$ induces an isomorphism
\begin{equation} \label{eq:Koszul_SO_isom}
K(\sum_j h_{1j}s_j, \ldots, \sum_j h_{rj} s_j) \cong K(s_1,\ldots,s_r).
\end{equation} 
A hermitian metric on $E$ induces a metric on $K(s)$: different $\wedge^k E$ are orthogonal and an orthonormal basis of $\wedge^k E$ is given by all elements $e_{i_1} \wedge \ldots \wedge e_{i_k}$, where $\{e_1,\ldots,e_{rk(E)}\}$ is an orthonormal basis of $E$. If $E$ is hermitian, then the above isomorphism is an isometry if $(h_{ij}) \in SU(r)$.

\bibliographystyle{amsplain}
\bibliography{refs} 

\providecommand{\bysame}{\leavevmode\hbox to3em{\hrulefill}\thinspace}
\providecommand{\MR}{\relax\ifhmode\unskip\space\fi MR }
\providecommand{\MRhref}[2]{%
  \href{http://www.ams.org/mathscinet-getitem?mr=#1}{#2}
}
\providecommand{\href}[2]{#2}
\begin{thebibliography}{10}

\bibitem{BergeronMillsonMoeglin}
Nicolas Bergeron, John Millson, and Colette Moeglin, \emph{{T}he {H}odge
  {C}onjecture and {A}rithmetic {Q}uotients of {C}omplex {B}alls}, to appear in
  Acta Mathematica.

\bibitem{BGV}
Nicole Berline, Ezra Getzler, and Mich{\`e}le Vergne, \emph{Heat kernels and
  {D}irac operators}, Grundlehren der Mathematischen Wissenschaften
  [Fundamental Principles of Mathematical Sciences], vol. 298, Springer-Verlag,
  Berlin, 1992. \MR{1215720 (94e:58130)}

\bibitem{BismutGilletSoule1}
J.-M. Bismut, H.~Gillet, and C.~Soul{\'e}, \emph{Analytic torsion and
  holomorphic determinant bundles. {I}. {B}ott-{C}hern forms and analytic
  torsion}, Comm. Math. Phys. \textbf{115} (1988), no.~1, 49--78. \MR{929146
  (89g:58192a)}

\bibitem{BismutInv90}
Jean-Michel Bismut, \emph{Superconnection currents and complex immersions},
  Invent. Math. \textbf{99} (1990), no.~1, 59--113. \MR{1029391}

\bibitem{BismutHLBCC}
\bysame, \emph{Hypoelliptic {L}aplacian and {B}ott-{C}hern cohomology},
  Progress in Mathematics, vol. 305, Birkh\"auser/Springer, Cham, 2013, A
  theorem of Riemann-Roch-Grothendieck in complex geometry. \MR{3099098}

\bibitem{BruinierFunke}
Jan~Hendrik Bruinier and Jens Funke, \emph{On two geometric theta lifts}, Duke
  Math. J. \textbf{125} (2004), no.~1, 45--90. \MR{2097357 (2005m:11089)}

\bibitem{FaltingsLecturesARR}
Gerd Faltings, \emph{Lectures on the arithmetic {R}iemann-{R}och theorem},
  Annals of Mathematics Studies, vol. 127, Princeton University Press,
  Princeton, NJ, 1992, Notes taken by Shouwu Zhang. \MR{1158661}

\bibitem{GreenGriffithsKerrBook}
Mark Green, Phillip Griffiths, and Matt Kerr, \emph{Mumford-{T}ate groups and
  domains}, Annals of Mathematics Studies, vol. 183, Princeton University
  Press, Princeton, NJ, 2012, Their geometry and arithmetic. \MR{2918237}

\bibitem{GriffithsRoblesToledo}
Phillip Griffiths, Colleen Robles, and Domingo Toledo, \emph{Quotients of
  non-classical flag domains are not algebraic}, Algebr. Geom. \textbf{1}
  (2014), no.~1, 1--13. \MR{3234111}

\bibitem{GriffithsSchmid}
Phillip Griffiths and Wilfried Schmid, \emph{Locally homogeneous complex
  manifolds}, Acta Math. \textbf{123} (1969), 253--302. \MR{0259958}

\bibitem{KudlaSplitting}
Stephen~S. Kudla, \emph{Splitting metaplectic covers of dual reductive pairs},
  Israel J. Math. \textbf{87} (1994), no.~1-3, 361--401. \MR{1286835}

\bibitem{KudlaOrthogonal}
\bysame, \emph{Algebraic cycles on {S}himura varieties of orthogonal type},
  Duke Math. J. \textbf{86} (1997), no.~1, 39--78. \MR{1427845 (98e:11058)}

\bibitem{KudlaCD}
\bysame, \emph{Central derivatives of {E}isenstein series and height pairings},
  Ann. of Math. (2) \textbf{146} (1997), no.~3, 545--646. \MR{1491448
  (99j:11047)}

\bibitem{KudlaMillson1}
Stephen~S. Kudla and John~J. Millson, \emph{The theta correspondence and
  harmonic forms. {I}}, Math. Ann. \textbf{274} (1986), no.~3, 353--378.
  \MR{842618 (88b:11023)}

\bibitem{KudlaMillson2}
\bysame, \emph{The theta correspondence and harmonic forms. {II}}, Math. Ann.
  \textbf{277} (1987), no.~2, 267--314. \MR{886423 (89b:11041)}

\bibitem{KudlaMillson3}
\bysame, \emph{Intersection numbers of cycles on locally symmetric spaces and
  {F}ourier coefficients of holomorphic modular forms in several complex
  variables}, Inst. Hautes \'Etudes Sci. Publ. Math. (1990), no.~71, 121--172.
  \MR{1079646 (92e:11035)}

\bibitem{QuillenChern}
Daniel Quillen, \emph{Superconnections and the {C}hern character}, Topology
  \textbf{24} (1985), no.~1, 89--95. \MR{790678 (86m:58010)}

\bibitem{SouleBook}
C.~Soul{\'e}, \emph{Lectures on {A}rakelov geometry}, Cambridge Studies in
  Advanced Mathematics, vol.~33, Cambridge University Press, Cambridge, 1992,
  With the collaboration of D. Abramovich, J.-F. Burnol and J. Kramer.
  \MR{1208731 (94e:14031)}

\end{thebibliography}

\vspace{1cm}

\author{\noindent Department of Mathematics \\
  University of Toronto \\
  40 St. George Street, BA 6290 \\ 
  Toronto, ON M5S 2E4, Canada \\
  e-mail: \texttt{lgarcia@math.toronto.edu}}

\end{document}